\documentclass{amsart}
\usepackage[T1]{fontenc}
\usepackage[english]{babel}
\usepackage{amsmath,amsthm,amssymb, latexsym,geometry,stackrel,enumerate}
\usepackage[all]{xy}
\usepackage{hyperref}

\usepackage[dvips]{graphicx}
\DeclareGraphicsExtensions{.bmp,.png,.pdf,.jpgg}
\usepackage[all]{xy}
\usepackage{verbatim}
\usepackage[mathcal]{euscript}
\usepackage{epsfig}
\usepackage{multicol}
\usepackage{color}
\usepackage{verbatim}
\usepackage{graphicx,float}

\date{\today}

\swapnumbers
\theoremstyle{plain}
\newtheorem{teo}{Theorem}[section]
\newtheorem{lema}[teo]{Lemma}
\newtheorem{prop}[teo]{Proposition}
\newtheorem{coro}[teo]{Corollary}

\newtheorem{convencion}[teo]{\bf{Convention}}

\theoremstyle{definition}
\newtheorem{defi}[teo]{Definition}
\newtheorem{obs}[teo]{Remark}
\newtheorem{obss}[teo]{Remarks}
\newtheorem{ejem}[teo]{Example}

\newcommand{\ts}[1]{\normalfont{\textsf{#1}}}
\newcommand{\K}{\ts k}

\renewcommand{\a}{\alpha}
\renewcommand{\b}{\beta}

\def\Hom{\mathop{\rm Hom_{\mathcal{C}_m}}\nolimits}
\def\P{\mathop{P_{p,q,m}}\nolimits}
\def\D{\mathop{\Delta}\nolimits}
\def\m{\mathop{(m+2)}\nolimits}
\def\End{\mathop{\rm End_{\mathcal{C}_m}}\nolimits}
\def\C{\mathop{\mathcal{C}_m}\nolimits}

\def\Ho{\mathop{\rm Hom}\nolimits}
\def\dim{\mathop{\rm dim}\nolimits}

\newcommand{\mor}[3]{$#1\colon #2 \to #3$}

\title[ $m$-cluster tilted algebras of type $\widetilde{\mathbb{A}}$]{ $m$-cluster tilted algebras of type $\widetilde{\mathbb{A}}$}

\author[V. Gubitosi]{Viviana Gubitosi}
\address{Instituto de Matem\'{a}tica y Estad\'{\i}stica Rafael Laguardia, Facultad de Ingenier\'{\i}a - UdelaR, Montevideo, Uruguay, 11200 }
\email{gubitosi@fing.edu.uy}

\keywords{$m$-cluster tilted algebras; gentle algebras}
\begin{document}
\maketitle

\begin{abstract}
 In this paper, we characterize all the finite dimensional algebras that are  $m-$cluster tilted algebras of type $\widetilde{\mathbb{A}}$. We show that these algebras are gentle and we give an explicit description of their quivers with relations.
\end{abstract}

\section*{Introduction} Cluster categories were introduced in \cite{Buan2006}, and simultaneously in \cite{CSS06} for the case $\mathbb{A}$,  in order to model the combinatorics of the cluster algebras of Fomin and Zelevinski \cite{FZ02}, using  tilting theory. The clusters correspond to the tilting objects in the cluster category.

Roughly speaking, given an hereditary finite dimensional algebra $H$ over an algebraically closed field, the cluster category $\mathcal{C}_H$ is obtained from the derived category $\mathcal{D}^b(H)$ by identifying the shift functor $[1]$ with the Auslander - Reiten translation $\tau$. By a result of Keller \cite{K05}, the cluster category is triangulated, and the same holds for  the category obtained from $\mathcal{D}^b(H)$ by identifying the composition $[m]:=[1]^m$ with $\tau$. The latter is called an $m$-cluster category, in which $m$-cluster tilting objects have been defined by Thomas, in \cite{Thomas2007}, who in addition showed that they are in bijective correspondence with the $m$-clusters associated by  Fomin and Reading to a finite root system in \cite{FR05}. The endomorphism algebras of the $m$-cluster tilting objects are called $m$-cluster tilted algebras or, in case $m=1$, cluster tilted algebras.

Caldero, Chapoton and Schiffler gave in \cite{CSS06} an interpretation of the cluster categories $\mathcal{C}_H$ in case $H$ is hereditary of Dynkin type $\mathbb{A}$ in terms of triangulations of the disc with marked points on its boundary, using an approach also present in \cite{FST08} in a much more general setting. This has been generalized by several authors. For instance, Baur and Marsh in \cite{BM08} considered $m$-angulations of the disc, modeling the $m$-cluster categories. In \cite{ABCP09}, Assem \emph{et al.} showed that cluster tilted algebras coming from triangulations of the disc or the annulus with marked points on their boundaries are gentle, and, in fact, that these are the only gentle cluster tilted algebras. The class of gentle algebras defined by Assem and Skowro\'nski in \cite{AS87} has been extensively studied in \cite{AH81, Avella-Alaminos2008, BB10, Buan2008, Murphy2010, SZ03}, for instance, and is particularly well understood, at least from the representation theoretic point of view. This class includes, among others, iterated tilted and cluster
tilted algebras of types $\mathbb{A}$ and $\tilde{\mathbb{A}}$, and, as shown in \cite{SZ03}, is closed under derived equivalence. In \cite{Murphy2010}, Murphy gives a description of the $m$-cluster tilted algebras of  type $\mathbb{A}$ in terms of quivers and relations. To do so, he works with the geometric realization of the $m$-cluster category of type $\mathbb{A}$   made by Baur and Marsh in \cite{BM08}. In \cite{GubBust} we give a classification up to derived equivalence of the $m$-cluster tilted algebras of  type $\mathbb{A}$, see \cite{BB10} for the case $m=1$.

In the present paper, we completely classify the $m$-cluster tilted algebras of  type $\tilde{\mathbb{A}}$ in terms of quivers and relations, using the geometric model proposed by Torkildsen in \cite{T12-arxiv}.


We now state the main result of this paper (for the definitions of the terms used, we refer the reader to section 7 below).

\subsection*{Theorem }\textit{ A connected algebra $A=\K Q/I$ is  a connected component of an $m$-cluster tilted algebra of type $\widetilde{\mathbb{A}}$ if and only if  $(Q,I)$ is a gentle bound quiver satisfying the following conditions: }

\begin{itemize}
  \item  [(a)]\textit{It can contain a non-saturated  cycle  $\widetilde{\mathcal{C}}$ in such a way that $A$ is an algebra with root $\widetilde{\mathcal{C}}$.}
  \item [(b)]\textit{If  $\widetilde{\mathcal{C}}$ is an oriented  cycle, then it must have at least one internal relation.}
  \item [(c)]\textit{If the quiver contains  more cycles, then all of them are  $m$-saturated.}
  \item [(d)]\textit{Outside of an $m$-saturated cycle it can contain  at most  $m-1$ consecutive relations.}
  \item [(e)]\textit{If there are internal relations in the root cycle, then  the number of clockwise oriented relations is equal modulo $m$  to the number of counterclockwise oriented.}
\end{itemize}

\medskip

It is important to observe that $m$-cluster tilted algebras  with $m\geq 2$ need not be connected and, as we shall see, their type is not uniquely determined, as we show below. This fact is essential in \cite{FT}.

\medskip
The paper is organized as follows: In section 1 we recall facts about gentle algebras, $m$-cluster tilted algebras and a geometric model of the $m$-cluster category of type $\widetilde{\mathbb{A}}$. We use this section to fix some notation. In sections 2 and 3 we establish some properties of $\m$-angulations  and  $m$-cluster categories that will be used in the sequel. In section 4   we study the $m$-cluster tilting objects. Sections 5 and 6 are devoted  to the  study of the bound quiver of an $m$-cluster tilted algebra of type  $\widetilde{\mathbb{A}}$.  We continue in section 7 with some consequences, among which, we prove that $m$-cluster tilted algebras of type $\widetilde{\mathbb{A}}$ are gentle.



\section{Preliminaries}
\subsection{Gentle algebras}

While we briefly recall some  concepts concerning bound quivers and algebras, we refer the reader to \cite{ASS06} or \cite{ARS95}, for instance, for unexplained notions.

Let \K \  be a commutative field. A \emph{quiver} $Q$ is the data of two sets, $Q_0$ (the \textit{vertices}) and $Q_1$ (the \textit{arrows}) and two maps \mor{s,t}{Q_1}{Q_0} that assign to each arrow $\a$ its \textit{source} $s(\a)$ and its \textit{target} $t(\a)$. We write \mor{\a}{s(\a)}{t(\a)}.  If $\b\in Q_1$ is such that $t(\a)=s(\b)$ then the composition of $\a$ and $\b$ is the path $\a\b$. This extends naturally to paths of arbitrary positive length. The \emph{path algebra} $\K Q$ is the $\K$-algebra whose basis is the set of all paths in $Q$, including one stationary path $e_x$ at each vertex $x\in Q_0$, endowed with the  multiplication induced from the composition of paths. In case $|Q_0|$ is finite, the sum of the stationary paths  - one for each vertex - is the identity.

If the quiver $Q$ has no oriented cycles, it is called \emph{acyclic}. A \emph{relation} in $Q$ is a $\K$-linear combination of paths of length at least $2$ sharing source and target.  A relation which is a path is called \emph{monomial}, and the relation is \emph{quadratic} if the paths appearing in it have all length $2$. Let $\mathcal{R}$ be a set of relations.
 Let $\langle Q_1 \rangle$ denote the two-sided ideal of $\K Q$ generated by the arrows, and $I$ be the one generated by $\mathcal{R}$. Then $I\subseteq \langle Q_1\rangle ^2$.
 The ideal $I$ is called \emph{admissible} if there exists a natural number $r\geqslant 2$ such that $\left\langle Q_1 \right\rangle^r \subseteq I$. The pair $(Q,I)$ is a \emph{bound quiver}, and associated to it is the algebra $A=\K Q/I$.
It is known that any finite dimensional basic and connected algebra over an algebraically closed field is obtained in this way, see \cite{ASS06}, for instance.

Recall from \cite{AS87} that an algebra  $A= \K Q/I$ is said to be \emph{gentle} if  $I=\langle \mathcal{R} \rangle $, with $\mathcal{R}$ a set of monomial quadratic relations such that :
\begin{enumerate}
 \item[G1.] For every vertex $x\in Q_0$ the sets $s^{-1}(x)$ and $t^{-1}(x)$ have cardinality at most two;
 \item[G2.] For every arrow $\a\in Q_1$ there exists at most one arrow $\b$ and one arrow $\gamma$ in $ Q_1$ such that $\a\b\not\in I$, $\gamma\a\not\in I $;
 \item[G3.] For every arrow $\a\in Q_1$ there exists at most one arrow $\b$ and one arrow $\gamma$ in $Q_1$ such that $\a\b\in I$, $\gamma\a\in I$.
\end{enumerate}

Gentle algebras are special biserial (see \cite{WW85}), and have  been extensively studied in several contexts, see for instance \cite{Avella-Alaminos2008, BB10, Buan2008, Murphy2010, SZ03}.

\subsection{$m$-cluster tilted algebras}

Let $H\simeq \K Q$ be an hereditary algebra such that $Q$ is acyclic. The derived category $\mathcal{D}^b(H)$ is triangulated, the translation functor, denoted by $[1]$, being induced from the shift of complexes. For an integer $n$, we denote by $[n]$ the composition of $[1]$ with itself $n$ times, thus $[1]^n = [n]$. In addition, $\mathcal{D}^b(A)$ has Auslander-Reiten triangles, and, as usual, the Auslander-Reiten translation is denoted by $\tau$.

Let $m$ be a natural number. The $m$-cluster category of $H$ is the quotient category $\mathcal{C}_m(H):=\mathcal{D}^b(H)/ \tau^{-1} [m]$ which carries a natural  triangulated structure, see \cite{K05}. In \cite{Buan2006} the authors showed that the $m$-cluster category
$\C(H)$ is \textit{Calabi-Yau of CY-dimension $m+1$} (shortly $(m+1)$-CY), that is,  there is a bifunctorial  isomorphism  $\rm Hom_{\C(H)}$$(X,\tau Y[1])\cong D \rm Hom_{\C(H)}$$(Y,X)$.

 Following \cite{Thomas2007} we consider \textit{$m$-cluster tilting objects}  in $\mathcal{C}_m(H)$ defined as objects satisfying the following conditions:
\begin{enumerate}
 \item $\rm Hom_{\mathcal{C}_m(H)}$$ (T,X[i]) = 0$ for all $i\in\{1,2,\ldots,m\}$ if and only if $X\in \ts{add}\ T$,
 \item $\rm Hom_{\mathcal{C}_m(H)}$$ (X,T[i]) = 0$ for all $i\in\{1,2,\ldots,m\}$ if and only if $X\in \ts{add}\ T$.
\end{enumerate}

The endomorphism algebras of such objects are called \emph{$m$-cluster tilted algebras of type $Q$}. In case $m=1$, this definition specializes to that of a cluster tilted algebra, a class intensively studied since its definition in \cite{BMR06}.



In \cite{ABCP09} it has been shown that cluster tilted algebras are gentle if and only if they are of type  $\mathbb{A}$ or $\tilde{\mathbb{A}}$. Using arguments similar to those of \cite{ABCP09}, Murphy showed in \cite{Murphy2010} that $m$-cluster tilted algebras, with $m\geq 2$, are also gentle  and  that the only cycles that may exist  are cycles of length $(m+2)$, briefly called $(m+2)$-cycles. Each of these cycles has full relations, that is, the composition of any two arrows on the cycle is zero. Perhaps the most noticeable differences between cluster tilted and $m$-cluster tilted algebras is that the latter need not to be connected,  that the quiver does not determine the algebras and that  there is no uniqueness of type like we will show in section 7.\\

%
%
%
%

Throughout the rest of this paper,  $\C$ denotes the $m$-cluster category  $\mathcal{C}_m(H)$ with $H$ an hereditary algebra of type $\tilde{\mathbb{A}}$ and we fix the following numbering and orientation  for the quiver  $\widetilde{\mathbb{A}}$ (or  $\widetilde{\mathbb{A}}_{p,q}$):

$$  \xymatrix@R=.3cm@C=.5cm  {  & \scriptstyle{ 0} \ar[dr] \ar[dl]& \\
               \scriptstyle{1}\ar[d] & & \scriptstyle{p+1} \ar[d] \\
                \scriptstyle{2} \ar[d] && \scriptstyle{p+2} \ar[d] \\
                \scriptstyle{\vdots} \ar[d] && \scriptstyle{\vdots} \ar[d] \\
                  \scriptstyle{p-1} \ar[dr] &&  \scriptstyle{p+q-1} \ar[dl] \\
                  & \scriptstyle{p} &}$$

where $p,q\geq 1$.

 Recall that  the Auslander-Reiten  quiver of $\mathcal{C}_m$ contains $m$ transjective components $\mathcal{S}^d$, $m$ tubular components  $\mathcal{T}^d_p$ of rank $p$,  $m$ tubular components $\mathcal{T}^d_q$ of rank $q$,  where $d\in \{0, \cdots, m-1\}$,  and infinitely many tubular components of  rank $1$.
%

\subsection{A geometric realization of the $m$-cluster category of type $\tilde{\mathbb{A}}$.} We follow \cite{T12-arxiv}. Let $m \geq 1$ and  $p, q \geq2$ be integers. Let $P_{p,q,m}$ be a
regular $mp$-gon, with a regular $mq$-gon at its center, cutting a hole in the interior of
the outer polygon. Denote by $P_{p,q,m}^0$ the interior between the outer and inner polygon. Label the vertices on the outer polygon $O_0,O_1,...,O_{mp-1}$ in the counterclockwise direction, and label the vertices of the inner polygon $I_0,I_1,...,I_{mq-1}$, in the clockwise direction.

Let $\delta_{i,k}$   (or  $\gamma_{i,k}$ ) be the path in the counterclockwise (or clockwise, respectively) direction from $O_i$ ( or $I_i$) to $O_{i+k-1}$ (or $I_{i+k-1}$, respectively) along the border of the outer (or inner) polygon, where $k$ is the number of vertices that the path runs through (including the start and end vertex). 

%

We define two paths to be \emph{equivalent} if they start in the same vertex,
end in the same vertex and they are homotopic. We call these equivalence classes
diagonals in $P_{p,q,m}$. Let $O_{i,t}$ denote the diagonals homotopic to $\delta_{i,t}$, and let $I_{i,t}$ be the diagonals homotopic to $\gamma_{i,t}$.

An \emph{$m$-diagonal} in $P_{p,q,m}$  of type $1$ is  a diagonal  between $O_i$ and $I_j$ with $i$  congruent to $j$ modulo $m$. An $m$-diagonal  of type $2$  (or type $3$) is a diagonal of the form $O_{i,km+2}$ (or $I_{i,km+2}$, respectively), with $k \geq 1$ and  $i\in \{0,\cdots, pm-1\}$ ( or $i\in \{0,\cdots, qm-1\}$, respectively).

We say that a set of $m$-diagonals \emph{cross} if they intersect in the interior $P_{p,q,m}$. A set of non-crossing $m$-diagonals that divides $P_{p,q,m}$ into $(m+2)$-gons is called an $(m+2)$-angulation. Torkildsen \cite[Prop. 4.2]{T12-arxiv} shows that the number of $m$-diagonals in any  $(m+2)$-angulation of $P_{p,q,m}$ is exactly $p+q$.

\subsubsection{The  quiver corresponding to an $\m$-angulation}

For any $\m$-angulation $\D$ of $\P$, Torkildsen \cite{T12-arxiv} defines a corresponding coloured quiver $Q_{\D}$ with $p+q$ vertices in the following way. The vertices are the $m$-diagonals. There
is an arrow between $i$ and $j$ if the $m$-diagonals bound a common $\m$-gon. The
colour of the arrow is the number of edges forming the segment of the boundary of
the $\m$-gon which lies between $i$ and $j$, counterclockwise from $i$. This is the
same definition as in \cite{Buan2009} in the Dynkin $\mathbb{A}$ case, and it is easy to see that such a
quiver satisfies the conditions described there for coloured quivers.

Let $Q_{\D}^0$ be  the  subquiver of the coloured quiver  associated to the $\m$-angulation $\D$ given by the arrows of colour $0$. The idea of the present  article is to show that all   $m$-cluster tilted algebras of  type $\tilde{\mathbb{A}}$ are given by an $\m$-angulation $\D$ of the  polygon $\P$ in such a way that the ordinary quiver of the algebra is the quiver   $Q_{\D}^0$ and the relations are also determined by the  $\m$-angulation.  This description allows us to show easily that  $m$-cluster tilted algebras of type $\tilde{\mathbb{A}}$ are gentle.\\

\subsubsection{The category of $m$-diagonals } Let $\alpha$ be an $m$-diagonal in $P_{p,q,m}$ and $s$ a positive integer. We define $\alpha[s]$ to be the $m$-diagonal obtained by rotating the outer polygon $s$ steps clockwise and the inner polygon $s$ steps counterclockwise. Also we have the opposite operation that we denote by $[-s]$. We set $\tau=[m]$.

We define the \emph{category of $m$-diagonals} $\mathcal{C}^m_{p,q}$
as the $\K$-linear additive category having as indecomposable objects the $m$-diagonals,  so the objects are the finite sets of $m$-diagonals. The space of morphisms between the indecomposable objects   $X$ and $Y$ is the vector space over $\K$ spanned by the elementary moves  (morphisms defined in such a way that they correspond to irreducible morphisms in $\C$) modulo certain mesh relations. See \cite[Section 6]{T12-arxiv} for more details.

Let $\alpha$ be an $m$-diagonal. If $\alpha$ is of type 1, we say that
$\alpha$ has \textit{level}  $d$ (with $0 \leq d <m$) if we can reach any
$m$-diagonal between $O_d$ and $I_d$ from $\alpha$ by applying a finite sequence of $\tau$ and elementary moves. If $\alpha$ is of type  2 (or 3), we say that $\alpha$ has level
 $d$  if we can reach the $m$-diagonal $O_{d+1,m+2}$ (or $I_{d+1,m+2}$, respectively)
from $\alpha$ by applying a finite sequence of  $\tau$ and elementary moves.

Given  the category of $m$-diagonals $\mathcal{C}^m_{p,q}$ we define the \emph{Auslander-Reiten quiver of $\mathcal{C}^m_{p,q}$} in the following way. The vertices are the indecomposable objects (that is,  the $m$-diagonals), and there is an arrow from the indecomposable object $\alpha$ to the indecomposable object $\beta$ if and only if there is an elementary move $\alpha \rightarrow \beta$. 

Following the same notation as in  \cite{T12-arxiv}, we denote by  $\mathcal{T}_p^d$ (or $\mathcal{T}_q^d$), the tube of rank $p$ ( or $q$, respectively) of degree  $d$ in the  $m$-cluster category $\C$ and  by  $\mathcal{S}^d$ the component in the Auslander-Reiten  quiver of $\C$ consisting of  objects of the form $\tau^sP[d]$ where $P$ is   projective. Similarly let $T^d_p$ (or $T^d_q$) be the component in the Auslander-Reiten  quiver of $\mathcal{C}^m_{p,q}$ containing the objects of type 2 (or 3, respectively) in level $d$ and $S^d$ the component consisting of $m$-diagonals of type 1 in level $d$.

Now, we want to define an additive functor $F : \mathcal{C}^m_{p,q} \rightarrow
\mathcal{C}_m$ in such a way that $F$ induces a quiver isomorphism between the Auslander-Reiten quiver of $\mathcal{C}^m_{p,q}$ and a subquiver of the Auslander-Reiten  quiver of $\C$.

We start defining a particular  $\m$-angulation  $\D_P$
of $\P$. This $\m$-angulation is such that the image
$F(\D_P)=\bigoplus_{i=0}^{p+q-1}P_i$,  where $P_i$ is the projective corresponding to the vertex $i$ in the quiver. Thus,  $\D_P=\{\alpha_0, \cdots, \alpha_{p+q-1}\}$ as in the figure below.


\begin{figure}[H]
\begin{center}
\hspace*{0cm}\includegraphics[scale=1.2]{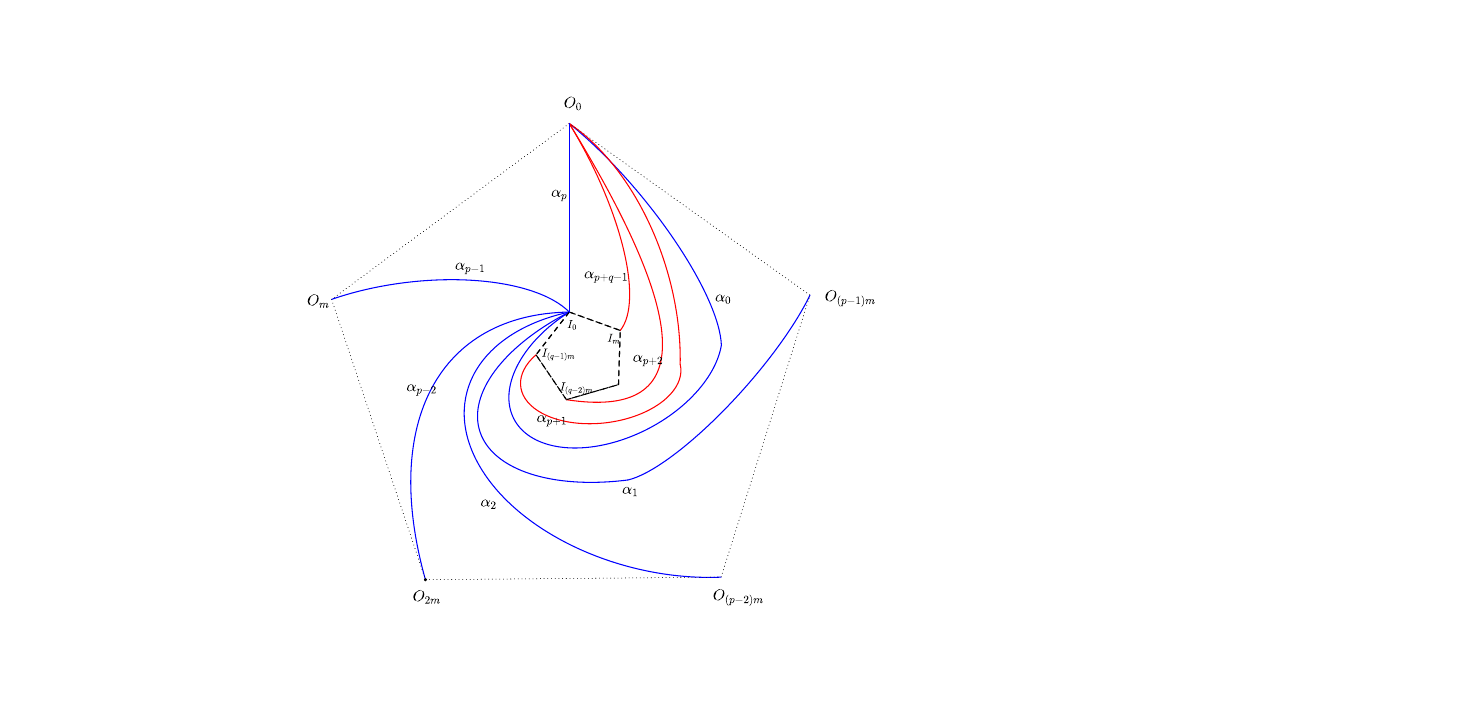}
\end{center}
\end{figure}

Defining $F(\tau^t \alpha_i[d])=\tau^t P_i[d]$ we get a bijection between the set of $m$-diagonals of type 1 and the set of indecomposable objects in the transjective components in the Auslander-Reiten quiver of $\C$.

 We denote the $p$  indecomposable objets on the mouth (the quasi-simples) of the tube $\mathcal{T}^d_p$ of rank $p$  by $M_i[d]$ (for $i\in \{0,\cdots, p-1\}$) where
$$M_0:=\begin{array}{c}  \scriptstyle{0} \\ \scriptstyle{p+1} \\ \scriptstyle{\vdots} \\ \scriptstyle{p} \end{array} \text{ \ \ \ \ \ \ \ and \ \ \ \ \ \ } M_k:=S_{p-k} \text { \ \  if \ \ } 1\leq k< p.$$

The other indecomposable objects that are in levels greater than 1 of the tube  are denoted by $Q_i^s[d]$  where  $s$ is the  quasi-length.  For a given $i$, a \emph{ray} is an infinite sequence of irreducible maps
 $$M_i[d] \rightarrow Q^2_i[d] \rightarrow Q^3_i[d] \rightarrow
Q^4_i[d] \rightarrow \hdots,$$

From now on, we also fix a notation for the diagonals of  type $2$ (or $3$) depending on the degree $d$ of the tube $T^d_p$ (or $T^d_q$, respectively) and the level $k$ that the diagonal has in the corresponding  tube. Then, the $p$  diagonals that are in the level  $k$ of the  tube $T^d_p$  are parametrized by $O_{mx-(d+1),km+2}$ with $x\in \{1,\cdots,p\}$ and $k\geq 1$.
Thus, we define $F(O_{mx-(d+1),m+2})=M_{x+1}[d]$ and for $k\geq 2$, $F(O_{mx-(d+1),km+2})= Q_{x+1}^{k}[d]$ (where we compute the sub-index always modulo $p$). Therefore we get a bijection between the set of $m$-diagonals of type 2 and the set of indecomposable objects in the tubular component $\mathcal{T}^d_p$ in the Auslander-Reiten quiver of $\C$.

We do similarly with the $m$-diagonals of type 3 and they correspond to objects in the tubes of rank $q$. Checking that elementary moves (arrows in the quiver of  $\mathcal{C}^m_{p,q}$ ) correspond to irreducible morphisms in $\C$ we have:

\begin{prop}\textnormal{\cite{T12-arxiv} }

\begin{enumerate}
\item The component $S^d$  in the Auslander-Reiten quiver of  $\mathcal{C}^m_{p,q}$ is isomorphic (via $F$) to the component $\mathcal{S}^d$ in the Auslander-Reiten quiver of the
$m$-cluster category $\mathcal{C}_m$ of type $\widetilde{\mathbb{A}}$,
for all integers $d$, with $0 \leq d < m$.
\item The components $T^d_p$ and $T^e_q$ of the quiver of   $\mathcal{C}^m_{p,q}$
are isomorphic (via $F$) to $\mathcal{T}^d_p$ and $\mathcal{T}^e_q$ respectively
in the  $m$-cluster category.

\end{enumerate}

\end{prop}

\section{Some results about  $\m$-angulations}

\begin{lema}\label{las 3 diag no pueden ser de tipo 1}

 Let $d_i$, $d_j$ and $d_k$ be    $m$-diagonals which are part of the same   $\m$-gon in the $\m$-angulation $\D$. Assume that they are arranged in such  a way that  $d_i$ and $d_j$ share a vertex of $\P$, $d_j$ and $d_k$ share another vertex of  $\P$ but $d_i$, $d_j$ and $d_k$ do not have vertices in common (See figure below). Then $d_i$, $d_j$ and $d_k$  cannot be simultaneously  $m$-diagonals of type 1.

\end{lema}

\vspace*{-.9cm}
\begin{figure}[H]
\begin{center}
$$\includegraphics[scale=.08 ]{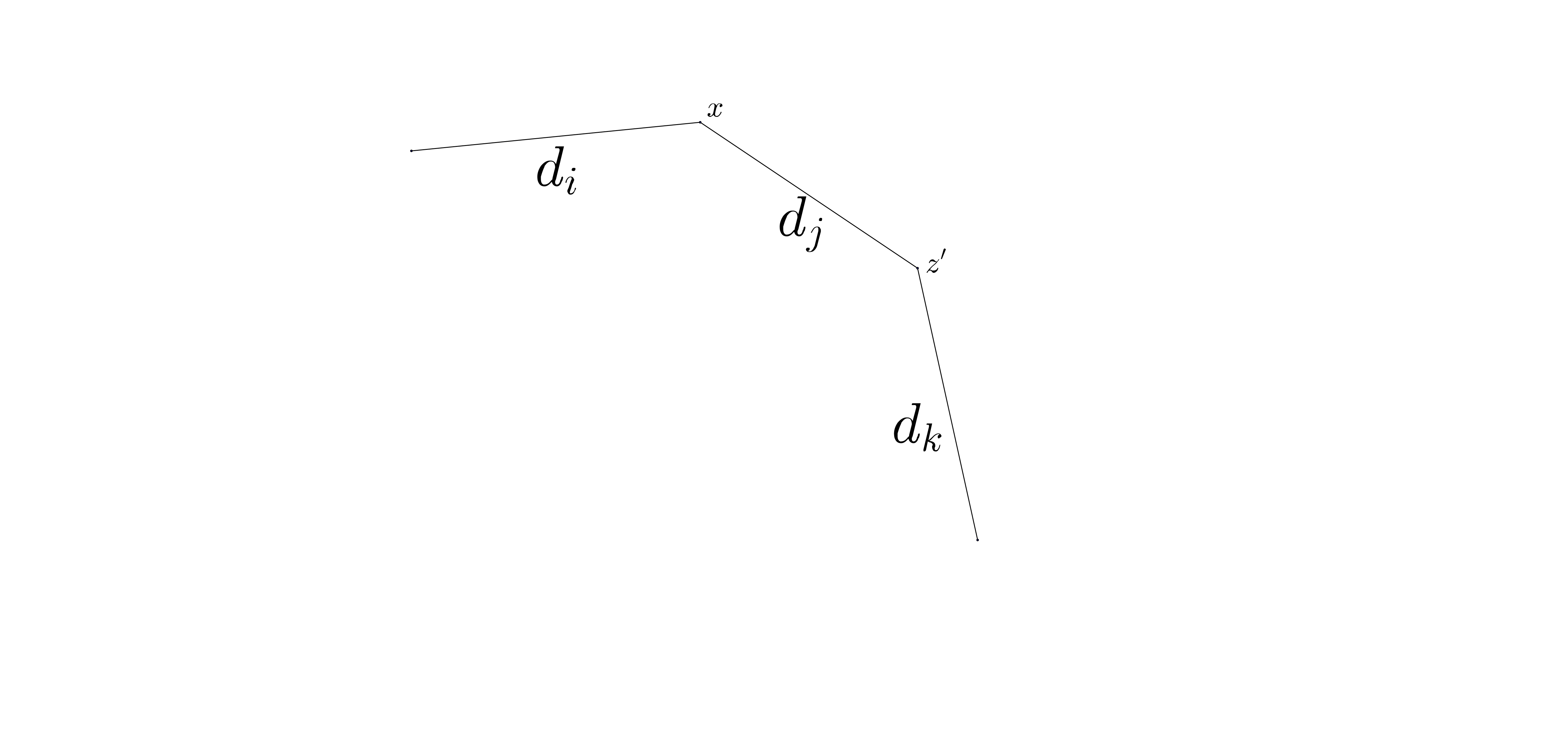}$$
\end{center}
\end{figure}

%

\begin{proof}
We assume $d_i\cap d_j=\{x\}$ with $x$ a vertex of the  outer $mp$-polygon. Since the three diagonals delimit the same  $\m$-gon, the diagonals $d_j$ and $d_k$ share a vertex that must not be   $x$. Then, $d_j\cap d_k=\{z'=t(d_j)\}$ with $z'$ in the inner  $mq$-polygon. Therefore $t(d_k)=z'$ and $s(d_k)=x'$  with  $x'\neq x$. We have two possibilities for  $d_k$  depending on the direction it takes.

\begin{figure}[H]
\begin{center}
\includegraphics[scale=0.3]{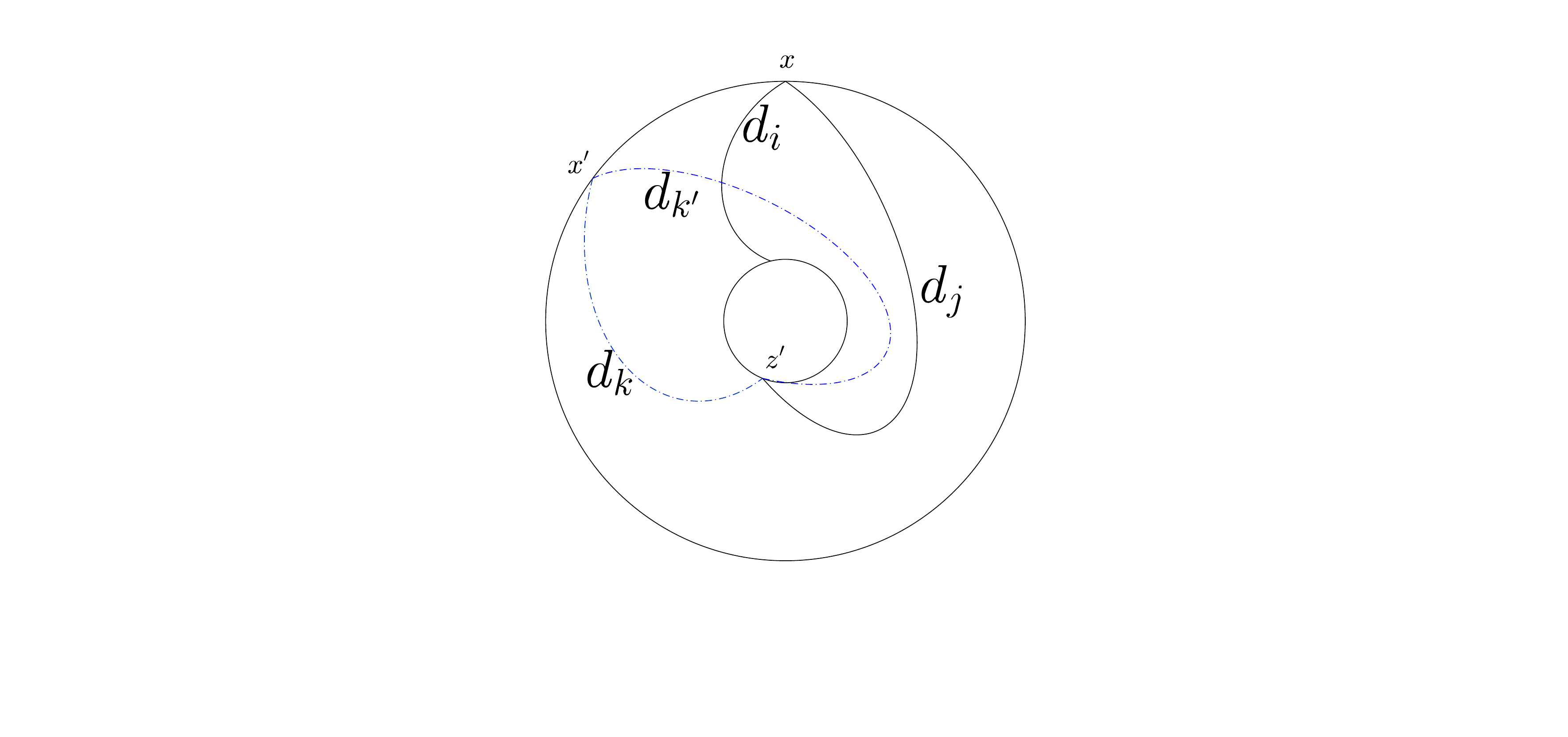}
\end{center}
\end{figure}

\vspace*{-1.5cm}
If $d_k$ is a diagonal between  $x'$ and $z'$ in the clockwise sense, it is clear that   $d_k\cap d_i \neq \emptyset$, which is impossible. If instead  $d_k$ is a  diagonal between $x'$ and $z'$ in the counterclockwise sense we will have an arrow $d_k\rightarrow d_j$ instead of an arrow $d_j\rightarrow d_k$ in  $Q_{\D}^0$. In consequence,  it is impossible for the three diagonals to be of type 1.

\end{proof}

\begin{lema}\label{final de una diagonal comienzo de la otra}
Let $d$ and $d'$  be   $m$-diagonals of type 2 (or 3). Assume that $d$ belongs to the component $T_p^{i}$  (or $T_q^{i}$ )  of the category of   $m$-diagonals $\mathcal{C}^m_{p,q}$. If the  source $s(d')$ of  $d'$ is equal to the target  $t(d)$ of $d$, then $d'$ belongs to the component  $T_p^{i-1}$ (or $T_q^{i-1}$, respectively).
\end{lema}

\begin{proof}
Since $d \in T_p^{i}$, $d=O_{mx-(i+1),km+2}$ with $x\in \{1,\cdots,p\}$ and $k\geq 1$.  Then, $s(d)=mx-(i+1)$ and $t(d)=mx-(i+1)+km+1=m(x+k)-i$. If $s(d')=t(d)$ we have  $d'=O_{m(x+k)-i,k'm+2} $ with $k'\geq1$. Therefore $d'\in T_p^{i-1}$ as claimed.
\end{proof}

An easy consequence of this lemma is the following corollary.

\begin{coro}
Let $\mathcal{C}$ be a component of the Aulander-Reiten quiver of the category $\mathcal{C}^m_{p,q}$ and let $d$ and $d'$ be two  $m$-diagonals of type 2 (or type 3) belonging to $\mathcal{C}$. Then the target of the diagonal  $d$ cannot be the source of the diagonal $d'$ reciprocally.
\end{coro}

\section{The category  $\C$}

  In this section we expose some easy results about  morphisms in $\C$ and we fix some notation.


\begin{lema}\label{morfismos entre componente C y componente C[i] con i>1 son cero}

 Let  $H$ be an hereditary algebra and $m>2$. Then $\Ho_{\mathcal{C}_m(H)}(M,N[2])=0$ for all  $M,N\in  mod H$.
\end{lema}

\begin{proof} By definition,

\begin{eqnarray*}
\Ho_{\mathcal{C}_m(H)}(M,N[2])  =  \bigoplus_{i\in \mathbb{Z}} \Ho_{D^b(H)}(M,F^iN[2]) = \bigoplus_{i\in \mathbb{Z}} \Ho_{D^b(H)}(M,\tau^{-i}N[im+2]).
\end{eqnarray*}

All the terms $\Ho_{D^b(H)}(M,\tau^{-i}N[im+2])$ with $im+2<0$ are zero because the objects $M$ and  $\tau^{-i}N$ belong to  $(modH)[0]$. 
In addition, all the terms $\Ho_{D^b(H)}(M,\tau^{-i}N[im+2])$ with $im+2\geq2$ are zero for  $H$ hereditary. To conclude, since $m>2$, we cannot have $im+2=0$ or $im+2=1$ and hence $\Ho_{\mathcal{C}_m(H)}(M,N[2]) =0$.
\end{proof}


\begin{lema}\label{morfismos entre T_p y T_q[1]}
Let $X$ be an object in the tubular component $\mathcal{T}^x_p$ and $Y$  an object in the tubular component $\mathcal{T}^x_q$, with $p\neq q$. Then $\Hom(X,Y[1])=0$.
\end{lema}

\begin{proof}
Using the  $m+1$-Calabi-Yau property of   $\C$ we get
\begin{eqnarray*}
\Hom(X,Y[1]) \cong   D\Hom(\tau^{-1}Y,X)  \cong  D\Hom(Y,\tau X)
= 0
\end{eqnarray*}


which is zero because there are no morphisms between elements in different tubes.
\end{proof}

\section{The $m$-cluster tilting objets in $\C$  }

\begin{defi}
Let   $\mathcal{C}$  be a  component of the  Auslander-Reiten quiver of  $\C$ and   $M$ an object in  $\mathcal{C}$. Then we define the following subsets:

$$M^r=\{N\in \mathcal{C} | \ \ \Hom(M,N)\neq
0 \}$$

$$M^l=\{N\in \mathcal{C} |  \ \  \Hom(N,M)\neq
0 \}$$

Given an object $M$ in  $\mathcal{C}$  the set  $M^r \setminus (\tau^{-1}M)^r$ determine two rays starting in  $M$. By convention we call $M^{ru}$ the ray  which is located above and $M^{dr}$ the ray  which is  below.

In the same way, the set $M^l \setminus (\tau M)^l$ determine two rays   ending in $M$. By convention we call $M^{lu}$ the ray  which is located above and $M^{dl}$ the ray  which is  below.\\
\end{defi}

\begin{convencion}
Let $P_p$ be the projective associated to the vertex $p$. We adopt the convention that $P_p^{dr}$ is the ray

  $$P_p\rightarrow P_{p-1}\rightarrow \cdots \rightarrow P_1 \rightarrow P_0$$

 and that  $P_p^{ur}$ is the ray

   $$P_p\rightarrow
P_{p+q-1}\rightarrow \cdots \rightarrow P_{p+1} \rightarrow P_0$$

\end{convencion}

\begin{obs}
\textnormal{
Observe that in the component of  $\C$ isomorphic  to $\mathcal{S}^0$ the
parallel rays to $P_p\rightarrow P_{p-1}\rightarrow \cdots
\rightarrow P_1 \rightarrow P_0$  are in  bijection (via  $F^{-1}$) with the diagonals that  share a vertex in the inner polygon. }\

\textnormal{  In the same way, the
parallel rays to $P_p\rightarrow P_{p+q-1}\rightarrow \cdots \rightarrow P_{p+1} \rightarrow P_0$ are in  bijection (via  $F^{-1}$)  with the diagonals that  share a vertex in the outer polygon.}
\end{obs}

In the following lemmas we are going to see the relationship between two indecomposables summands  of an  $m$-cluster tilting object  $T$ if both are in the same component.\\

\begin{lema}\label{componente transyectiva}
Let $T$ be  an  $m$-cluster tilting object and let  $M,N$ be two  objects in a transjective component. Then,

\begin{enumerate}
  \item If $M\in add(T)$, then $\tau^{-i}M \notin add(T)$ for all $i\in \mathbb{N}^*$.
  \item If $M,N\in add(T)$ and there is a nonzero morphism $M\rightarrow N$, then $N\in M^{ur} \cup M^{dr}$.
\end{enumerate}
\end{lema}

\begin{proof}
(1) In the  $m$-cluster category $[m]=\tau$ and consequently, for all
 $i\in \mathbb{N}^*$ we have: $$\Hom(M,\tau^{-i}M[m])\cong \Hom(M,\tau^{-(i-1)}M)\neq 0.$$
Thus,  $\tau^{-i}M\notin add(T)$.\\
(2) Since there is a nonzero  morphism $M\rightarrow N$, we know that
 $N\in M^r$. Suppose that $N\notin M^{ur}\cup M^{dr}$. Then, $N=\tau^{-i}N'$ with
  $N'\in M^{ur}\cup M^{dr}$ or $N=\tau^{-i}M$, $i\in \mathbb{N}^*$. By 1,  we cannot have $N=\tau^{-i}M$. Thus, $N=\tau^{-i}N'$ with $N'\in M^{ur}\cup M^{dr}$ and
$$\Hom(M,\tau^{-i}N'[m])\cong \Hom(M,\tau^{-(i-1)}N')\neq 0.$$
In consequence  $N=\tau^{-i}N'\notin add(T)$, which contradicts our assumption.
\end{proof}


\begin{defi}
Let   $\mathcal{C}$  be a  tubular component of the  Auslander-Reiten quiver of  $\C$ and   $\mathcal{N}$ a subset of  $\mathcal{C}$. For any $r\in \mathbb{N}^*$, we define the  subset $\mathcal{N}_{\leq  r}$ by:

$$\mathcal{N}_{\leq  r}=\{N\in \mathcal{N} | \textnormal{  $N$  belongs to the first $r$ levels of the tube} \}$$

\end{defi}

\begin{lema}
Let $T$ be an  $m$-cluster tilting object  and let  $M\in add(T)$ be an object in the tubular component of rank $p$.  Then,  $M$ belongs to the first  $p-1$  levels of the tube.
\end{lema}

\begin{proof}
Observe that if $d_M$ is a  diagonal of type $2$ such that $F(d_M)=M$ then  $d_M$ has the form $O_{*,km+2}$,  where $k$ is the level of the  diagonal in the tube  $T_p$. If $k\geq p$, $km+2\geq pm+2$.   Then, the diagonal has a self-intersection and cannot be part of an  $\m$-angulation. In consequence $M$ cannot be  an  $m$-cluster tilting object summand.
\end{proof}

\begin{lema}\label{sumandos de T en un tubo}
Let $T$ be  an  $m$-cluster tilting object  and let   $M,N$ be two  objects in a  tubular component of rank $p$. If $M,N\in add(T)$ and there is a nonzero  morphism  $M\rightarrow N$, then  $N\in (M^{ur} \cup M^{dr})_{\leq p-1}$.

\end{lema}

\begin{proof}
   As in the proof of lemma \ref{componente transyectiva} we show that $N\in M^{ur} \cup M^{dr}$. It remains to prove that $N$ belongs to the first  $p-1$ levels of the  tube, but this follows from the previous lemma.
\end{proof}

Now, we want to study the  relations  between two summands of an $m$-cluster tilting object  when both are in different components.\\

\begin{lema}\label{morfismos entre P' y algun tau de P[1]}
Let $T$ be  an  $m$-cluster tilting object. Assume that  $P_a \in add(T)$ and $\tau^rP_b[1]\in add(T)$. Then $\Hom(P_a,\tau^rP_b[1])=0$.
\end{lema}

\begin{proof}
Since  $P_a, \tau^rP_b[1] \in add(T)$ and $T$ is an  $m$-cluster tilting object we have
$$0=\Hom(\tau^rP_b[1],P_a[1])\cong
\Hom(P_b,\tau^{-r}P_a)$$
 and then $\tau^{-r}P_a \notin P_b^r$. Otherwise,  if  $0\neq \Hom(P_a,\tau^rP_b[1])$ using the  $(m+1)$-CY  property of $\C$ we obtain  $0\neq D\Hom(\tau^rP_b, \tau P_a)\cong D\Hom(P_b,\tau^{1-r}P_a)$ and $\tau(\tau^{-r}P_a) \in P_b^r$, a contradiction.
\end{proof}

To continue we need the following definitions.\\

\begin{defi}
Let  $\mathcal{C}$ be a tubular component of the  Auslander-Reiten quiver of  $\C$ and let   $M$ be an object in   $\mathcal{C}$. Then we define the following subsets:

$$M^+=\{N\in \mathcal{C} | \ \  \Hom(M,N)\neq
0 \}_{\leq p-1}$$
$$M^-=\{N\in \mathcal{C} | \ \ \Hom(N,M)\neq
0 \}_{\leq p-1}$$
\end{defi}

In the same way we can define the same subsets if
$\mathcal{C}$ is a component  of the category of  $m$-diagonals $\mathcal{C}^m_{p,q}$.\\

\begin{defi}
Let  $M$ be an  object belonging to the mouth of a  tube  of rank $p$ in $\mathcal{C}_m$ or $\mathcal{C}^m_{p,q}$. We define  the \emph{cone} $\widehat{M}$  of $M$ to be the set  $$\widehat{M}=\left (\bigcup_{N\in M^{dr}}N^{dl}\right )_{\leq p-1} $$
\end{defi}

\begin{lema}\label{M^}Let $S_i$ be the simple associated to the vertex $i$ of the quiver and  let $M$ be an indecomposable  object   in the same  tube as  $S_i$. Then:
\begin{enumerate}
  \item $\text{If  } i\in \{1,\cdots, p-1 \},  \text{  } \dim \Hom(P_i,M)=\left\{
              \begin{array}{ll}
                1, & \hbox{\text{if }  $M\in \widehat{S_i}$ ;} \\
                0, & \hbox{\text{otherwise}.}
              \end{array}
            \right.$
  \item $\text{If  } i\in\{0,p, p+1,\cdots, p+q-1\},  \text{   }\dim \Hom(P_i,M)=\left\{\begin{array}{ll}
                                                             1, & \hbox{\text{if  }$M\in \widehat{M_0}$;} \\
                                                             0, & \hbox{otherwise.}
                                                           \end{array}
                                                         \right.$
\end{enumerate}

\end{lema}

\begin{proof} To prove this it is enough to see that the  simple $S_i$ appears just once  as a composition factor of  $M$ if $M\in \widehat{S_i}$ and does not appear outside  the cone  $\widehat{S_i}$. For the  quasi-simple module $M_0$, the simples $S_0, S_p$ and $S_j$ (with
$j\in\{p+1, \cdots, p+q-1\} $)  appear once in the  cone
$\widehat{M_0}$ and do not appear in the complement. The others  simples  do not appear.
\end{proof}

\begin{lema}\label{elementos del tubo q tienen morf desde un proy}
Let $T$ be an  $m$-cluster tilting object and  $P_b, N\in add(T)$ such that  $N$  is an  object in the tubular component $\mathcal{T}_p$. Suppose that $\Hom(P_b,N)\neq 0$, then  $N\in
M_*^{dl}$ where $$*=\left\{
                     \begin{array}{ll}
                       p-b, & \hbox{if $b\in\{1,\cdots, p-1\}$;} \\
                       0, & \hbox{otherwise.}
                     \end{array}
                   \right.
$$
\end{lema}

\begin{proof}
Since $\Hom(P_b,N)\neq 0$, by remark \ref{M^}, we know that  $N\in \widehat{M_*}$ with $*$ as above. In addition, since  $P_b, N\in add(T)$, we have  $0=\Hom(N,P_b[1])\cong
 D \Hom(P_b,\tau N)$. Then  $\tau N\in \widehat{M_*}^c$ and $N\in
(\tau^{-1}\widehat{M_*})^c=(\widehat{M_{*-1}})^c$. In consequence
$N\in \widehat{M_*}\cap (\widehat{M_{*-1}})^c=M_*^{dl} $.
\end{proof}

\begin{lema}\label{tubo tauideproy[1]}
Let $T$ be an  $m$-cluster tilting object and $N, \tau^rP_b[1]\in add(T)$ such that  $N$  is an object in the tubular component $\mathcal{T}_p$. Assume that $\Hom(N,\tau^rP_b[1])\neq 0$,
then $N\in M_{*-1+r}^{dr}$ with $*$ as in the previous lemma (and we compute $*-1+r$ modulo p).
\end{lema}

\begin{proof}
Since  $0 \neq \Hom(N,\tau^rP_b[1])\cong D\Hom(\tau^rP_b,\tau N)
\cong D \Hom(P_b,\tau^{1-r}N)$ we have  $\tau(\tau^{-r}N)\in
\widehat{M_*}$. In the other hand, since $\tau^rP_b[1], N\in add(T)$ we have
$$0=\Hom(\tau^rP_b[1], N[1])\cong \Hom(P_b,\tau^{-r}N).$$
 Then
$\tau^{-r}N\in (\widehat{M_*})^c$ and consequently  $\tau^{-r}N\in
(\widehat{M_*})^c \cap \tau^{-1}\widehat{M_*}=M_{*-1}^{dr} $ and $N\in
\tau^rM_{*-1}^{dr}=M_{*-1+r}^{dr}$.
\end{proof}

In order to simplify the notation in the following lemma we introduce the next definition. \\

\begin{defi}
Let  $B$ be  an object in a tube of rank $p$ wich lies on the level
$v\in \{1, \cdots, p-1\}$ of the  tube. The index   $\alpha(B)$ of $B$ is defined by  $\alpha(B)=p-1-v$.

\end{defi}

\begin{lema}\label{morfismos entre un tubo y el siguiente}
Let $T$  be an  $m$-cluster tilting object and  $B,D \in add (T)$ such that $B\in \mathcal{T}^i_p$ and $D\in \mathcal{T}^{i+1}_p$. Then,
$\Hom(B,D)\neq 0$ if and only if $D\in
\bigl ((\mathcal{B}^-(\tau B)[1])^{dl}\bigr)_{\leq \alpha(B)}$,    and the level  $v$ of $B$ is different from $p-1$.

\end{lema}


\begin{proof}
 Since $D\in \mathcal{T}^{i+1}_p$,  there is $D'\in \mathcal{T}_p^i$ such that $D=D'[1]$. On the other hand,   $B,D \in add (T)$ implies that   $$0=\Hom(D'[1],B[1])\cong
\Hom(D',B)$$ and  $$ 0=\Hom(D'[1],B[2])\cong \Hom(D',B[1])\cong D
\Hom(B, \tau D').$$ Thus,  $D'\in ( B^-\cup (\tau^{-1}
B)^+)^c$. If in addition we want that $$ \Hom(B,D)\cong \Hom(B,D'[1])\cong
 D\Hom(D',\tau B)$$ be non-zero, we get
 $D'\in (\tau B)^-\setminus (B^-\cup (\tau^{-1} B)^+)=\bigl ((\mathcal{B}^-(\tau B)[1])^{dl}\bigr)_{\leq \alpha(B)}$.
 \end{proof}

To each diagonal $d$ in a tubular component of $\mathcal{C}^m_{p,q}$, we  associate two diagonals (which may coincide) in the mouth of the tube according to the following definition.\\

\begin{defi}
Let  $d$ be a  diagonal of type $2$ or $3$. We define  $\mathcal{B}(d)$ (or $\mathcal{B}^{-}(d)$ ) as the unique  diagonal that belongs to the intersection of $d^{ur}$ (or $d^{ul}$, respectively) and the mouth of the tube.    \

In the same way we define $\mathcal{B}(T)$ and $\mathcal{B}^{-}(T)$  if  $T$ is an  object in a tubular component of $\C$.
\end{defi}

The following proposition (and its dual) show   the importance of the objects that we have just defined.\\

\begin{prop}\label{morfismo a  la boca del tubo}
Let $d'$ be a diagonal of type $1$ and let  $d$  be a  diagonal of
type $2$ (or $3$). If $\Hom(F(d'),F(d))=0$, then
$\Hom(F(d'),F\mathcal{B}(d))=0$.

\end{prop}

\begin{proof}
It is enough to show that we can factorize the morphism
$F(d')\rightarrow F(\mathcal{B}(d))$ through the  morphism $F(d')\rightarrow F(d)$.
We proceed inductively on  the level  $k$
of the diagonal  $d=O_{*,km+2}$ (or its  image $F(d)$ ) on the tube $T_p$ (or $\mathcal{T}_p$, respectively). If
$k=1$ the diagonal is in the mouth of the tube and thus  $\mathcal{B}(d)=d$.
Assume now the statement holds true for   $k\leq n$ and let $d$ be a diagonal on the level  $n+1$. We label the elements on the ray  $d^{ur}$
$$d\stackrel{b}{\rightarrow} B^{n-1}(d)\stackrel
{j_{n-1}}{\rightarrow} B^{n-2}(d)\stackrel{j_{n-2}}{\rightarrow}
\cdots \stackrel{j_2}{\rightarrow} B^1(d)\stackrel{i}{\rightarrow}
\mathcal{B}(d)$$ (where the index  $s$ is the level on the tube less $1$). We have the following situation in the tube:

$$\xymatrix@R=.3cm@C=.3cm{ \ar@{.}[rrrrr]& & & & &\tau \mathcal{B}(d)\ar[dr]^{g} \ar@{.}[rr] & & \mathcal{B}(d) \ar@{.}[r]&\\
&&&&&& B^1(d) \ar[ur]_i & &\\
&&&&& \ar[ru]^{j_{2}}&&&\\
&& d''\ar[dr]^a \ar@{-->}[uuurrr]^{f'} && \ar@{--}[ru]&&&&\\
&\tau B^{n-1}(d)\ar[ru] \ar[dr] && B^{n-1}(d) \ar[ru]^{  j_{n-1}}&&&&&\\
&& d \ar[ru]_{b      } &&&&&& }$$

Set $j=j_2\circ \cdots \circ j_{n-1}$.
By the induction hypothesis there is   $c:d'\rightarrow B^{n-1}(d)$ such that the diagram

$$\xymatrix@R=.5cm@C=.2cm{d' \ar[rr] \ar@{.>}[dr]_{\exists c  \ \ }&& \mathcal{B}(d) \\
&B^{n-1}(d) \ar[ru]_{i\circ j}&}$$

commutes.\\
In addition, since  $0\longrightarrow \tau B^{n-1}(d) \longrightarrow d''\oplus d\stackrel{\tiny{\left(
                                                  \begin{array}{cc}
                                                    a & b \\
                                                  \end{array}
                                                \right)
}} {\longrightarrow} B^{n-1}(d)$  is an almost split sequence,  there is a  morphism $\tiny{\left(
                 \begin{array}{c}
                   h \\
                   h' \\
                 \end{array}
               \right)
}:d'\rightarrow d''\oplus d$ such that  $c=\tiny{\left(
                                                  \begin{array}{cc}
                                                    a & b \\
                                                  \end{array}
                                                \right)
} \circ \tiny{\left(
                 \begin{array}{c}
                   h \\
                   h' \\
                 \end{array}
               \right)
}$. Moreover, if we apply the induction hypothesis to the diagonal  $d''$, we obtain   $f:d'\rightarrow d''$
 such that the diagram

  $$\xymatrix@R=.4cm@C=.4cm{d' \ar[rr] \ar@{.>}[dr]_{\exists f  }&& \tau \mathcal{B}(d) \\
& d'' \ar[ru]_{f'}&}$$

commutes. Then, we have the following commutative diagram:

$$\xymatrix@R=.8cm@C=.8cm{ &&& d' \ar[d]^c \ar[dl]_{\tiny{\left(
                 \begin{array}{c}
                   h \\
                   h' \\
                 \end{array}
               \right)}}&\\
0 \ar[r] & \tau B^{n-1}(d)\ar[r] & d''\oplus d \ar[d]_{\tiny{\left(
                                                                      \begin{array}{cc}
                                                                        f' & 0 \\
                                                                        0 & Id \\
                                                                      \end{array}
                                                                    \right)}
} \ar[r]_{\tiny{\left(
                                                  \begin{array}{cc}
                                                    a & b \\
                                                  \end{array}
                                                \right)
}} & B^{n-1}(d) \ar[d]^j \ar[r] & 0 \\
&& \tau \mathcal{B}(d)\oplus d \ar[r] _{\tiny{\left(
                                                  \begin{array}{cc}
                                                    g & j\circ b \\
                                                  \end{array}
                                                \right)
}}& B^1(d) \ar[d]^i& \\
&&& \mathcal{B}(d) &
}$$

and the factorization   \ \  $\xymatrix@R=.8cm@C=.5cm{d' \ar[rr]^{i\circ j\circ c} \ar[dr]_{\tiny{\left(
                 \begin{array}{c}
                   h \\
                   h' \\
                 \end{array}
               \right)} }&& \mathcal{B}(d) \\
& d''\oplus d \ar[ru]_{ \  \  \ \ \ \ i\circ \tiny{\left(
                                                  \begin{array}{cc}
                                                    g & j\circ b \\
                                                  \end{array}
                                                \right)
} \circ \tiny{\left(
                                                                      \begin{array}{cc}
                                                                        f' & 0 \\
                                                                        0 & Id \\
                                                                      \end{array}
                                                                    \right)}= \tiny{\left(
                                                  \begin{array}{cc}
                                                    i\circ g \circ f' & i\circ j\circ b \\
                                                  \end{array}
                                                \right)}} &}$

which is actually the factorization \ \  $\xymatrix@R=.8cm@C=.8cm{d' \ar[rr] \ar[dr]_{h'}&& \mathcal{B}(d) \\
&d \ar[ru]_{i\circ j \circ b}&}$  \ \  since  $i\circ g =0$.

\end{proof}

\begin{obs}
Note that the converse is not true.  Let $A$ be the algebra  of type $\widetilde{\mathbb{A}}_{2,2}$ given by the quiver:

$$\xymatrix@R=.3cm@C=.5cm{  & 0  \ar[dr] \ar[dl]& \\
1\ar[rd] && 3 \ar[ld] \\
& 2 & }$$

Then, one of the tubes of rank $2$ of the Auslander-Reiten quiver of mod  $A$ is the following

$$\xymatrix@R=.08cm@C=.3cm{ \tiny{\begin{array}{c}
               0 \\
               3\\
               2
             \end{array} } \ar[rd] \ar@{.}[ddd] &&   1  \ar[dr] && \tiny{\begin{array}{c}
               0 \\
               3\\
               2
             \end{array} }  \ar@{.}[ddd] \\
             &  \tiny{\begin{array}{ccc}
               0  && \\
               3 && 1\\
               &2&
             \end{array} }  \ar[ur] \ar[dr] && \tiny{\begin{array}{ccc}
                 &0& \\
               3 && 1\\
               2 &&
             \end{array} } \ar[ru] \ar[dr]  & \\
             \ar[ru]&& \ar[ru]&& \\
             &&&&\\
              &&&&}$$

It is clear that $\Hom(2,1)=0$  but  $\Hom(2, \tiny{\begin{array}{ccc}
               0  && \\
               3 && 1\\
               &2&
             \end{array} })\neq 0$.
\end{obs}

\section{The ordinary quiver of the algebra $\End(F(\D))$ }

If $\D=\{d_1, \cdots, d_{p+q}\}$ is an  $\m$-angulation, then we denote by   $F(\D):=F(d_1)\oplus \cdots \oplus F(d_{p+q})$ the corresponding $m$-cluster tilting object.\\

We write $Q_{\D}^0$ the  subquiver of the coloured quiver  associated to the $\m$-angulation $\D$ given by the arrows of colour $0$. Now, the idea is to compare the  quiver $Q_{\D}^0$  with the ordinary quiver of the endomorphism  algebra $\End(F(\D))$.\\

\begin{defi} Let $x$ be a vertex in the outer polygon. We define  $d_{x\curvearrowright}$ as the set of all the  $m$-diagonals of type 1 starting at the vertex  $x$.
\end{defi}

\begin{lema} Let $x$ be a vertex of the outer polygon (then $0\leq x\leq mp-1$) such that $x\equiv 0$ (mod $m$). Then there  exists $y\in \{1,\cdots,p\}$ such that $x=m(p-y)$ and:

\begin{enumerate}
  \item   The diagonal $\alpha_{p-y}$ belongs to  $d_{x\curvearrowright}$,
and   $d_{x\curvearrowright}$ belongs to the component $S^0$.

  \item   $F(d_{x\curvearrowright})=\{\tau ^kP_{(p-y)-j}$ $ |$ $k\equiv j$ (mod $p$), $j\in \{-y,\cdots,p-y\} \}$ $\bigcup \{ \tau^{-y+kp}P_{p+j} | k\in Z, j\in \{0, \cdots, q\}\} $.
\end{enumerate}

\end{lema}

\begin{proof}
The existence of  $y$ is clear. Next, for 1 we simply observe that the diagonal $\alpha_{p-y}$  (if $ y\in \{1,\cdots,p\}$) is by definition a  diagonal starting at  $O_{m(p-y)=x}$ and finishing at  $I_0$ and it belongs to the component   $S^0$. Then $F(d_{x\curvearrowright})$ is the following  infinite ray  containing the object  $F(\alpha_{p-y})=P_{p-y}$:

\begin{figure}[H]
\includegraphics[scale=.8]{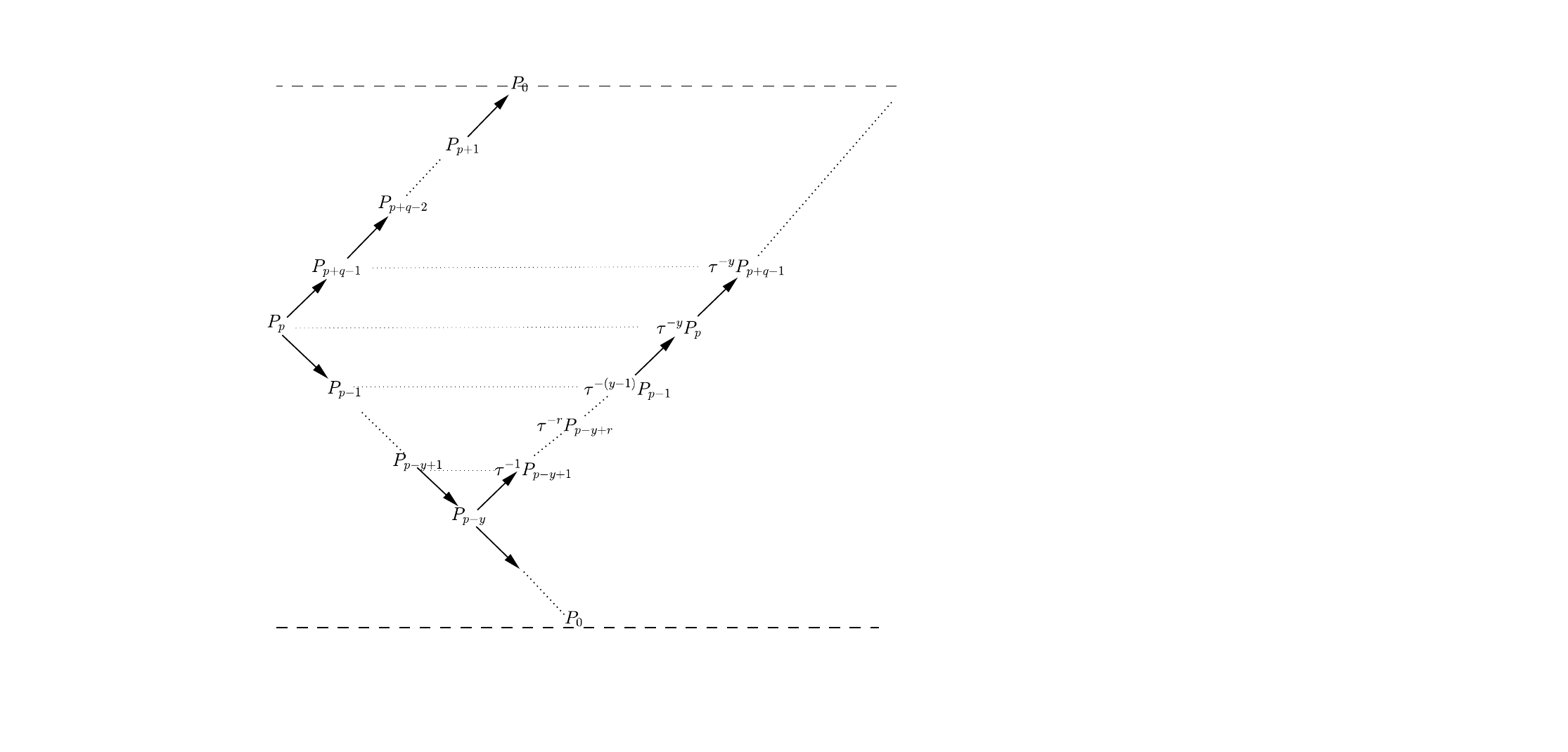}
\vspace{-1cm}
\caption{The ray $F(d_{x\curvearrowright})$ } \label{rayo}
\end{figure}

Consequently, $F(d_{x\curvearrowright})=\{\tau ^kP_{(p-y)-j}$ $|$ $k\equiv j$ (mod $p$), $j\in \{-y,\cdots,p-y\}$ \} $\bigcup \{ \tau^{-y+kp}P_{p+j} $$|$$ k\in \mathbb{Z}, j\in \{0, \cdots, q\}\} $.
\end{proof}

\begin{obs}
In general, if $x\equiv m-i$ (mod $m$) with  $i\in \{0,\cdots,m-1\}$ then   there is  $y\in \{1,\cdots,p\}$ such that  $x=m(p-y)-i$ and:

\begin{enumerate}

  \item   $d_{x\curvearrowright}$ belongs to the component $S^i$.
  \item   $F(d_{x\curvearrowright})=\{\tau ^kP_{(p-y)-j}[i]$ $|$ $k\equiv j$ (mod $p$), $j\in \{-y,\cdots,p-y\}$ \} $\bigcup \{ \tau^{-y+kp}P_{p+j}[i] \ \  |  k\in \mathbb{Z}, j\in \{0, \cdots, q\}\} $.
\end{enumerate}
\end{obs}
\qed

Now, we are ready to show the main result of this section.

\begin{prop} Let $\D$ be an  $\m$-angulation of $\P$.
The quiver $Q_{\End(F(\D))}$ associated  to the algebra  $\End(F(\D))$  is equal to the quiver   $Q_{\D}^0$ associated to the  $\m$-angulation $\D$.

\end{prop}

\begin{proof}
If $d_i$ is a diagonal in the $\m$-angulation of $\P$, we denote by $T_i$ the indecomposable object of $\C$ such that $F(d_i)=T_i$.\\

If two diagonals $d_1$ and $d_2$ share a vertex, there is a path between  $d_1$  and $d_2$ in  $Q_{\Delta}^0$. Assume that the path is oriented from  $d_1$ to $d_2$.\\

To begin, suppose that $d_1$ and $d_2$  are part of the same  $\m$-gon and assume that both diagonals are of the same   type $1$ (or $2$ or $3$) sharing a vertex. Then $d_1$ and $d_2$ lie in the same component  $S^d$ (or $T^d_p$ or $T^d_q$, respectively).
Thus, since there is an arrow   $d_1\rightarrow d_2$ and the two diagonals do not intersect, we know that  $d_2\in d_1^{ur}\cup d_1^{dr}$. According to the isomorphism between the components $S^d$  and $\mathcal{S}^d$ we have $T_2\in T_1^{ur}\cup T_1^{dr}$. Then, it is clear that there is a  morphism  $T_1\rightarrow T_2$ in  $\C$.\\
Now,  let us see the case where $d_1$ and $d_2$ are in two different $\m$-gons. Since both share a vertex of $\P$ there must be a third diagonal $d_3$ having the same vertex in common. Hence, we have the arrows $d_1\rightarrow d_3 \rightarrow d_2$ in $Q_{\Delta}^0$.
Assume that   $d_3$ is the same type as  $d_1$ and $d_2$.  Since we have the arrows   $d_1\rightarrow d_3 \rightarrow d_2$ we know that
$T_3\in T_1^{ur}\cup T_1^{dr}$ and $T_2\in T_3^{ur}\cup T_3^{dr}$. Suppose   (without loss of generality) that  $T_2\in T_1^{ur}$. Then, if $T_3\in T_1^{dr}$ we get that $T_2\notin T_3^{ur}\cup T_3^{dr}$ which is impossible. In consequence, $T_2$ and $T_3$ belong to  $T_1^{ur}$. Thus $T_1,T_2$ and $T_3$ lie in the same ray of the component  $\mathcal{S}^d$ and clearly there are   morphisms $T_1\rightarrow T_3 \rightarrow T_2$ in  $\C$.\\
We continue with the case where  $d_1$ is a diagonal of  type $1$ and $d_2$ is a diagonal of  type $2$ (or $3$),  both being part of the same  $\m$-gon.
In order to simplify the notation we assume that all diagonals are in degree  $0$. Then, we can write
$d_1=d_{xz}$ where $x=s(d_1)$, $z=t(d_1)$ and $x\equiv 0$ (mod $m$) and $d_2=O_{my-1,km+2}$ with
$y\in \{1,\cdots, p\}$ and $k\in \{1,\cdots, p-1\}$. Since there is an arrow   $d_1\rightarrow d_2$,
we know that  $d_2$ follows $d_1$   and  $t(d_2)=x$. Then, $t(d_2)=my-1+km+1\equiv x$ (mod $mp$)
 and so  $m(k+y)\equiv x$ (mod $mp$). In addition, $x \equiv  0$ (mod $m$) (since $d_1\in S^0$), then  $x=my'$ with $y'\in \{0,\cdots, p-1\}$.
 Combining these we obtain  $m(k+y-y') \equiv 0$ (mod $mp$) and so $k+y \equiv y'$ (mod $p$). In consequence  $\mathcal{B}(d_2)=O_{m(y+k-1)-1,m+2}=O_{m(y'-1)-1,m+2}$ and

$$F(\mathcal{B}(d_2))= M_{y'}=\left\{
             \begin{array}{ll}
                \begin{array}{c}  \scriptstyle{0} \\ \scriptstyle{p+1} \\ \scriptstyle{\vdots} \\ \scriptstyle{p}  \end{array} , & \hbox{if $y'\equiv 0$ (mod $p$);} \\
                \\
                S_{p-y'}, & \hbox{if $y'\neq 0$ (mod $p$).} \end{array}\right. $$

We want to prove that there is  a nonzero morphism  $T_1\rightarrow T_2$ in $\C$. By   proposition \ref{morfismo a  la boca del tubo} it suffices to show that there is a nonzero morphism  $T_1\rightarrow M_{y'}$.

Since  $T_1\in F(d_{x\curvearrowright})=\{\tau^kP_{p-y'-i} | k\equiv i \text{(mod $p$)} \}\cup \{\tau^{-y'+kp}P_{p+j}| k\in\mathbb{Z}, j\in\{0,\cdots, q\}\}$, we have the following possibilities:

\begin{itemize}
  \item If $T_1\in \{\tau^kP_{p-y'-i}| k\equiv i \text{(mod $p$)}\}$:

\begin{eqnarray*}
\Hom(\tau^kP_{p-y'-i},M_{y'})& \cong &\Hom(P_{p-y'-i},\tau^{-k}M_{y'}) \\
& = & \Hom(P_{p-y'-i},M_{k+y'})\\
& = & \Hom(P_{p-y'-i}, S_{p-k-y'}) \neq 0
\end{eqnarray*}

 because $p-y'-k\equiv p-y'-i$. \\

  \item If $T_1\in \{\tau^{-y'+kp}P_{p+j}| k\in\mathbb{Z}, j\in\{0,\cdots, q\}\}$:

\begin{eqnarray*}
\Hom(\tau^{-y'+kp}P_{p+j},M_{y'})& \cong &\Hom(P_{p+j},\tau^{y'-kp}M_{y'}) \\
& = & \Hom(P_{p+j},M_{y'-y'+kp})\\
& = & \Hom(P_{p+j},M_0)\neq 0
\end{eqnarray*}

because there is a monomorphism  $P_{p+j}\hookrightarrow M_0$.

\end{itemize}

We conclude that, in all the cases that we have seen above, there is a  morphism $T_1\rightarrow T_2$ in $\C$.\\
Now, let us see the case where  $d_1$ and  $d_2$ are in different $\m$-gons. As before, since both share a vertex of $\P$ there must be a third diagonal $d_3$ having the same vertex in common. Then, we have arrows $d_1\rightarrow d_3 \rightarrow d_2$ in  $Q_{\Delta}^0$.

At first, suppose that  $d_3$ and  $d_2$ are  diagonals of the same type sharing a vertex. Then, if   $x:=t(d_2)$, we also  have  $x=t(d_3)$ because, if we had  $x=s(d_3)$, then it would be  impossible for the diagonal $d_3$ to be in the middle of $d_1$ and $d_2$. Therefore  $d_2$ and $d_3$ share the   final vertex and consequently  $d_2\in d_3^{ur}$. Assume that $F(d_1)=P_a$ (otherwise we apply  $\tau$ as many times as necessary). Then, since there is a nonzero morphism $T_1=P_a \rightarrow T_3$, by lemma \ref{elementos del tubo q tienen morf desde un proy} we obtain $T_3\in
M_*^{dl}$, where $*= 0$ if $a\in \{p, p+1, \cdots, p+q-1, 0\}$ or $*=
p-a$ if $a\in \{1, \cdots, p-1\}$. Since $d_2\in d_3^{ur}$
we have $T_2, T_3 $ both lie in the same ray $M_*^{dl}$.

If $a\in \{1, \cdots, p-1\}$ the ray  $M_*^{dl}$ is

$$\cdots \rightarrow \tiny\begin{array}{ccccc}
                         & 0 &      &  &  \\
                       1 &   & p+1 &  &  a \\
                        &   & \vdots&  & \vdots\\
                         &   & p+q-1  &  & p-1 \\
                         &   &      & p &
                     \end{array} \rightarrow \begin{array}{ccccc}
                         & 0 &      &  &  \\
                       1 &   & p+1 &  &  a \\
                       2 &   & \vdots&  & \vdots\\
                         &   & p+q-1  &  & p-1 \\
                         &   &      & p &
                     \end{array}\rightarrow \begin{array}{ccccc}
                         & 0 &      &  &  \\
                        &   & p+1 &  &  a \\
                        &   & \vdots&  & \vdots\\
                         &   & p+q-1  &  & p-1 \\
                         &   &      & p &
                     \end{array} $$ $$ \tiny \rightarrow \begin{array}{c}
                                               a \\
                                               a+1 \\
                                               \vdots \\
                                               p-1
                                             \end{array} \rightarrow
\cdots \rightarrow \begin{array}{c}
                                               a \\
                                               a+1 \\
                                               \end{array} \rightarrow a $$

 and it follows easily from this that the  composition  $P_a\rightarrow T_3 \rightarrow T_2$ is nonzero. Since, by remark \ref{M^}, $\dim(\Hom(P_a,T_2))=1$
we deduce that the  morphism  $P_a\rightarrow T_2$ factors through $T_3$.
The computation is completely analogous if $a\in \{p, p+1, \cdots, p+q-1,
0\}$.\\
Finally,  suppose that   $d_3$ is a  diagonal of the same  type as $d_1$ (
type $1$). We want to show  that the  morphism $T_1 \rightarrow T_2$ factors through $T_3$. Since $d_1$ is a  diagonal of type $1$ we can assume that $T_1=P$ with $P$ projective in   mod $H$. We know that there is a morphism $d_1\rightarrow d_3$ then $d_3\in d_1^{ur}$ and equivalently $T_3\in T_1^{ur}$. If $T_1=P_{p+z}$ with $z\in \{0,\cdots , q-1\}$ we have that $T_3$ must be $P_{p+z'}$ with $ 0<z'<z$. If instead,  $T_1=P_{p-y}$ with $y\in\{1,\cdots, p\}$ we get that $T_3=\tau^{-r}P_{p-y+r}$ with $r\in \{1, \cdots, y-1\}$ or $T_3=\tau^{-y}P_{p+x}$ with $x\in \{0,\cdots, q-1\}$. See figure \ref{rayo}.

In all the cases  the morphism $T_1\rightarrow T_3$ is the inclusion of $T_1 $ in $T_3$, then if $f$ is the nonzero morphism $T_3\rightarrow T_2$, the composition $T_1\rightarrow T_3\rightarrow T_2$ is the restriction $f\mid _{T_1}$ which clearly is nonzero. Since, by  remark \ref{M^}, $\dim \Hom(T_1,T_2)=1$ we deduce that $f\mid _{T_1}$ is the morphism $T_1\rightarrow T_2$ as we wanted to show.
\end{proof}

\section{The quiver with relations associated to an  $\m$-angulation}

In light of the previous results, we know how to find the ordinary quiver of an    $m$-cluster tilted algebra of  type $\widetilde{\mathbb{A}}$. Now, our problem is to find  the relations in such a  quiver. \

We start by defining an ideal  $I_{\D}$ associated to an  $\m$-angulation $\D$ of $\P$ as follows:

 \begin{defi}\label{definicion de I_Delta}

 Let $i, j, k$  be the vertices in   $(Q^0_{\D})_0$ associated to the diagonals  $d_i, d_j$ and $d_k$ respectively. Given consecutive arrows $i \stackrel{\alpha}{\rightarrow} j \stackrel{\beta}{\rightarrow} k $ in   $Q^0_{\D}$ we define the path $\beta \alpha$ to be zero if the diagonals   $d_i, d_j$ and  $d_k$ are in the same   $\m$-gon in the   $\m$-angulation  $\D$. Let    $I_{\D}$ denote the ideal in the path algebra  $kQ_{\D}^0$ generated by all such  relations.

\end{defi}

Note that our definition agrees with the one given in \cite{Murphy2010} for the  case $\mathbb{A}$.\\

\begin{prop}\label{el carcaj con relaciones de End(F(delta))}
The algebra $kQ_{\D}^0/I_{\D}$ is isomorphic to the  $m$-cluster tilted algebra $\End(F(\D))$.
\end{prop}

\begin{proof}
We first prove that all the relations in the ideal  $I_{\Delta}$ are also relations in $\End(F(\D))$.
Suppose that  $d_i$, $d_j$ and $d_k$ are  $m$-diagonals which are part of the same  $(m+2)$-gon   arranged in such a way  that  $d_i$ and $d_j$ share a vertex of   $\P$, $d_j$ and $d_k$ share a vertex of  $\P$ but  $d_i$, $d_j$ and $d_k$ have no common vertex.

By the definition of  $Q_{\D}^0$ we have arrows  $d_i\stackrel{\alpha}{\rightarrow} d_j$
 and $d_j\stackrel{\beta}{\rightarrow} d_k$ in  $Q_{\D}^0$ which are in  bijection with the arrows  $T_i\rightarrow T_j$ and $T_j\rightarrow T_k$ between the  indecomposable factors of  $T:=F(\Delta)$. We want to show that the composition $T_i\rightarrow T_j\rightarrow T_k$  is zero in $\C$.\\
The proof will be divided into four cases, based on the type of the  diagonals. Let us first observe that,   by lemma \ref{las 3 diag no pueden ser de tipo 1},  the three diagonals $d_i, d_j$ and  $d_k$ cannot be simultaneously  of  type 1. Then, we have the following cases:\\

\begin{itemize}
  \item [(a)] \underline{The three  diagonals are of  type $2$ (or $3$)}: We have the following sub-cases.\\

\begin{itemize}
  \item [1.] The sources of $d_j$ and $d_k$ are the same: Since $d_i$, $d_j$ and $d_k$  are part of the same  $(m+2)$-gon, it is impossible to have   $d_i\cap d_j= s(d_j)=s(d_k)$. Then $d_i\cap d_j= t(d_j)$.

     Since the sources of  $d_j$ et $d_k$ coincide, the two diagonals are in the same tube (let us suppose  $T_p^r$). Moreover, since there is  a nonzero morphism $d_j\rightarrow d_k$ we have $d_k\in d_j^{dr}$ and then, by lemma \ref{final de una diagonal comienzo de la otra}, the diagonal $d_i$ belongs to the tube  $T_p^{r-1}$.   Applying lemma \ref{morfismos entre un tubo y el siguiente} to the induced morphism $T_i\rightarrow T_j$  we obtain $T_j\in \bigl((\mathcal{B}^-(\tau T_i)[1])^{dl}\bigr)_{\leq\alpha(T_i)}$. If  $\Hom(T_i,T_k)\neq 0$, by the same lemma, we see that
      $T_k$ belongs to the same set as $T_j$. But, if this is the case, we cannot have $T_k\in T_j^{dr}$ as we saw above. Consequently,  the composition $T_i\rightarrow T_j\rightarrow T_k$ is zero.\\


  \item [2.] The target of  $d_k$ is equal to the   source of $d_j$: Since there is a nonzero  morphism  $d_i\rightarrow d_j$ we have $t(d_j)=t(d_i)$ or $s(d_i)=t(d_j)$. In the first case, it is clear that the three diagonals cannot be part of the same  $\m$-gon. Then $s(d_i)=t(d_j)$. Suppose that $d_k\in T^{i}_p$, then since $t(d_k)=s(d_j)$ lemma \ref{final de una diagonal comienzo de la otra} implies  $d_j\in T^{i-1}_p$. The same lemma applied to the diagonals $d_i$ and $d_j$ gives $d_i\in T^{i-2}_p$. We want to show that $\Hom(T_i,T_k)=0$.
      If $m>2$ let $T_k'\in \mathcal{T}^{i-2}_p$ be such that $T_{k}:=T_k'[2]$, then  $\Hom(T_i,T_k)=\Hom(T_i,T_{k}'[2])=0$ by  lemma \ref{morfismos entre componente C y componente C[i] con i>1 son cero}. If instead $m=2$ the two diagonals $d_i$ and $d_k$ belong to the same tube, say $T^0$, and $d_j$ belongs to the tube $T^1$. Thus, the only way to have a nonzero  morphism $d_i\rightarrow d_k$ is that both diagonals share a vertex. That is,  $s(d_i)=s(d_k)$ or $t(d_i)=t(d_k)$. Suppose $s(d_i)=s(d_k)$, then since $t(d_j)=s(d_i)$ we have $t(d_j)=s(d_k)$. Recall that we are in the case where  $s(d_j)=t(d_k)$, in consequence the diagonals $d_j$  and $d_k$ share the same extremities and then  it is not possible to have one diagonal $d_i$ of the same  type so that the three  diagonals delimit the same $\m$-gon. If $t(d_i)=t(d_k)$ we conclude similarly that the diagonals $d_j$  and $d_i$ share the same extremities and thus  it is not possible to have one diagonal $d_k$ of the same  type so that the three  diagonals delimit the same $\m$-gon. Therefore it is not possible to have a nonzero morphism $T_i \rightarrow T_k$.\\

  \item [3.] The targets of  $d_j$ and $d_k$ are equal: Let  us see  the possibilities for  $d_i$. If  $t(d_i)=s(d_j)$ the three diagonals  cannot be part of the same  $\m$-gon. Then $s(d_i)=s(d_j)$ and consequently $T_j\in T_i^{dr}$. Since the three diagonals are in the same  tube then by lemma \ref{sumandos de T en un tubo}, to prove that $\Hom(T_i,T_k)=0$  it suffices to show that $T_k\notin T_i^{ur}\cup T_i^{dr}$. This is due to the facts that $T_j\in T_i^{dr}$ and   $t(d_j)=t(d_k)$ implies  $T_k\in T_j^{ur}$.\\

\end{itemize}

  \item [(b)] \underline{There is one  diagonal of each type}:  Since $d_i\cap d_j$ is a vertex in the inner (or outer) polygon   and  $d_j\cap d_k$ is a vertex in the outer (or inner, respectively) polygon , the diagonal $d_j$ must be the diagonal of type $1$. Assume $d_j\in S^z$, $d_i\in T^x_p$ is the   diagonal of type $2$ and $d_k \in T^y_q$ is the  diagonal of type $3$. Since there are arrows $d_i\rightarrow d_j \rightarrow d_k$ we have $z=y=x+1$. Then, if we take  $T_{k'}\in \mathcal{T}^x_p$ such that $T_k=T_{k'}[1]$, by lemma \ref{morfismos entre T_p y T_q[1]} we obtain  $\Hom(T_i,T_k)=\Hom(T_i,T_{k'}[1])=0$.\\

  \item [(c)] \underline{Two diagonals are of type $1$ and the other one is of   type $2$ (or $3$)}: We have the following sub-cases.\\

\begin{itemize}
  \item [1.] $d_i$ and $d_j$ are of type $1$ and $d_k$ is of type $2$ (or $3$):
   Let $d_i\cap d_j=\{x\}$. Assume  $x$ belongs to the outer polygon. The existence of a morphism $d_j\rightarrow d_k$ with  $d_k $ a diagonal of type $2$, implies $d_j\cap d_k=\{y\}$ with $y$ also in the outer polygon. Since $d_j$ is a  diagonal of type $1$, it has just one vertex in the outer polygon,  and thus  $x=y$ and the vertex $x$ is the source or the target  of $d_k$. Consequently  $d_i\cap d_j \cap d_k=\{x\}$ and the three  diagonals cannot be  part of the same   $\m$-gon. Therefore $x$ belongs to the inner  polygon, both diagonals of type $1$ have target  $x$ and the   diagonals $d_j$ and $d_k$ share the vertex  $y$ from the outer polygon. Hence $d_j\in d_i^{dr}$, that is,   $T_j\in T_i^{dr}$. Without loss of generality, since $T_i$ belongs to the  transjective component $\mathcal{S}^0$ we can assume that $T_i=P[0]$( $=P$ by abuse of notation) with $P$   projective in   mod $H$ (for, otherwise, we  apply  $\tau$ as many times as necessary to $T_i$ until we find a projective). Then,  $T_i\in P_p^{ur}\cup P_p^{dr}$ where $P_p$ is the projective associated with the vertex $p$ of the quiver.\\
   If  $T_i\in P_p^{dr} $ then  $T_i=P_a$ with  $a\in \{1,\cdots,p\}$ and $T_j=P_b$ with $0\leq b<a<p$. Since there is a nonzero  morphism $T_j\rightarrow T_k$, lemma \ref{elementos del tubo q tienen morf desde un proy} implies  $T_k\in S_b^{dl}$. If the  composition $T_i\rightarrow T_k$ is not zero, the same lemma applied to $T_i$ gives $T_k\in S_a^{dl}$. Clearly $S_a^{dl}\neq S_b^{dl}$ and consequently $\Hom(T_i,T_k)=0$.

  If instead  $T_i\in P_p^{ur}\setminus \{ P_p \}$, $T_i=P_{p+z}$    (with $z\in \{1, \cdots, q-1\}$). As  $T_j\in T_i^{dr}$ we have  $T_j=\tau^{-r}P_{p+z+r}$ with $r>0$ and $1\leq z \leq q-1$.  If the composition $T_i\rightarrow T_j\rightarrow T_k$ is not zero, then  by lemma \ref{elementos del tubo q tienen morf desde un proy}, we obtain $T_k\in M_0^{dl}$.  Since $0\neq \Hom(T_j,T_k)=\Hom(\tau ^{-r}P_{p+z+r}, T_k)\cong \Hom(P_{p+z+r},\tau^rT_k)$  the same lemma gives  $T_k\in \tau^{-r}M_0^{dl}=M_r^{dl}$. Then, $M_r^{dl}=M_0^{dl}$ which implies  $r\equiv 0$ (mod $p$) and equivalently $r=kp$ with $k\in \mathbb{Z}$. The diagonal $d_i$ is in bijection with the projective $P_{p+z}$ and the diagonal $d_j$ is  in  bijection with $\tau ^{-r}P_{p+z+r}$, then $d_i$ is a  diagonal between the vertices $O_0$ and $I_x$  and  $d_j=\tau^{-r}d$ with  $d$ a diagonal between the vertices $O_0$ and $I_y$  and   $y> x$.  Then $d_j=\tau^{-r}d=d[-rm]=d[-kpm]$  and consequently $d_j$ is a diagonal  between the vertices $O_0$ and $I_y$ and it is not possible to have a diagonal $d_k$ of type $2$ between both.\\

  \item [2.] $d_i$ is of type $2$ (or $3$) and  $d_k$ and  $d_j$ are of type $1$: It is the case dual to the one considered in 1 above. \\

  \item [3.] $d_i$ and $d_k$ are both of type $1$ and $d_j$ is of  type $2$ (or $3$): Suppose $d_i\cap d_j=\{x\}$ and $d_j\cap d_k=\{y\}$ with $x,y$ in the outer  polygon  in such a way that  $s(d_j)=y$ and $t(d_j)=x$. In particular  $x\neq y$ and $x=y+km+1$ (mod $mp$) with $k\geq 1$.  Assume $d_i\in S^r$, then  $x\equiv m-r$ (mod $m$) and hence  $y\equiv x-1\equiv m-(r+1)$ (mod $m$). Thus  $d_k\in S^{r+1}$ and $d_j\in T^{r}_p$. Since $d_i\in S^r$ and $d_k\in S^{r+1}$  there are  $s,s'\in \mathbb{Z}$ and $P,P'$ projective objects (in  mod $H$) such that  $T_i=\tau^s P[r]$  and  $T_k=\tau^{s'}P'[r+1]$. Then:

      \begin{eqnarray*}
      \Hom(T_i,T_k) & = & \Hom(\tau^s P[r],\tau^{s'}P'[r+1])\\
      & \cong  & \Hom(\tau^s P,\tau^{s'}P'[1])\\
      & = & \Hom(P, \tau^{s-s'}P'[1])=0
      \end{eqnarray*}

by lemma \ref{morfismos entre P' y algun tau de P[1]}.\\

\end{itemize}

  \item [(d)] \underline{Two diagonals are of type $2$ (or $3$) and the third one is of  type $1$}: \  \
       First of all, let us  see that $d_j$ cannot be the diagonal of type $1$. If this is the case, since there are arrows $d_i\rightarrow d_j$  and $d_j\rightarrow d_k$ the three  diagonals have to share a vertex of the outer polygon (inner if $d_i,d_k$ are of type $3$), but in this situation it is not possible that the three diagonals  delimit the same $\m$-gon.   Then, assume $d_i$ and $d_j$ are of type $2$ and $d_k$ is of type $1$. If   $s(d_i)=s(d_j)$ the possible  diagonals $d_k$ of type $1$ either have an intersection with  $d_i\cup d_j$ or do not delimit the same  $\m$-gon. (See figure 2).

       \begin{figure}[H]\label{caso d}
       \begin{center}
       \hspace*{2cm}\includegraphics[scale=.6]{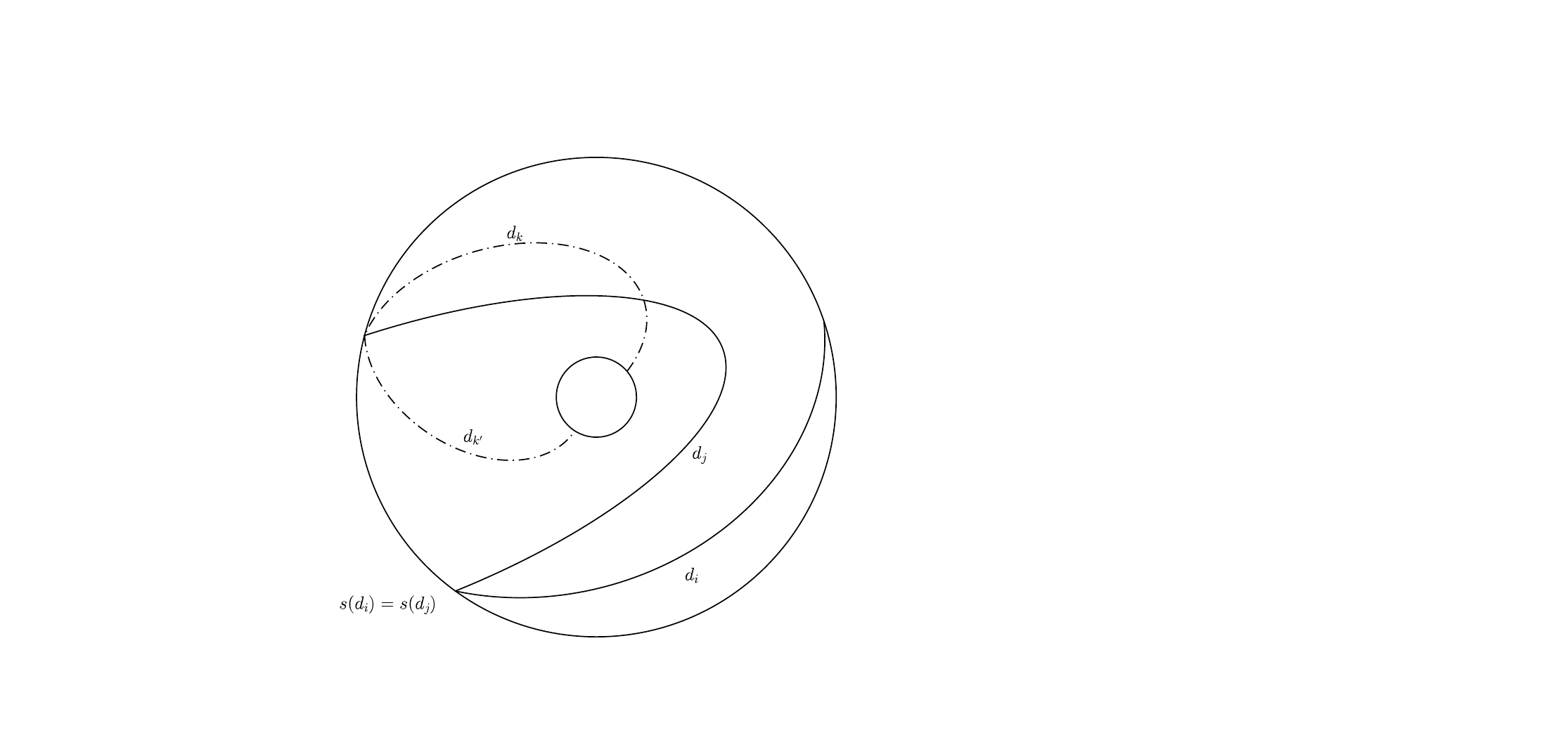}
       \vspace*{-1cm}
       \caption{The possible $d_k$ if  $s(d_i)=s(d_j)$. }
       \end{center}
       \end{figure}

      Hence,  $t(d_j)=s(d_i)$ and if $d_j\in T^r$ then $d_i\in T^{r-1}$ by lemma \ref{final de una diagonal comienzo de la otra}. In addition  $d_k\in S^{r+1}$.  If $m>2$ let $d_k'\in S^{r-1}$  be a diagonal such that  $d_k=d_k'[2]$. Consequently, lemma \ref{morfismos entre componente C y componente C[i] con i>1 son cero} implies $\Hom(T_i,T_k)\cong \Hom(T_i,T_k'[2])=0$. If $m=2$ we can assume $d_i\in T^0$, $d_j\in T^1$ and $d_k\in S^0$. Thus

      \begin{eqnarray*}
      \Hom(T_i,T_k) & =  &  \bigoplus_{j\in \mathbb{Z}} \Ho_{D^b(H)}(T_i,F^jT_k)
       =  \bigoplus_{j\in \mathbb{Z}} \Ho_{D^b(H)}(T_i,\tau^{-j}T_k[2j])
      \end{eqnarray*}

  All terms $\Ho_{D^b(H)}(T_i,\tau^{-j}T_k[2j])$ with $j< 0$ are zero because the objects $T_i$ and  $\tau^{-j}T_k$ are in (mod $H$ )$[0]$ and there are no backward morphisms in the derived category $ D^b$ (mod $H$). If $2j\geq2$ all terms $\Ho_{D^b(H)}(T_i,\tau^{-j}T_k[2j])$ are zero since $H$ is hereditary. Finally, if $j=0$ the term $\Ho_{D^b(H)}(T_i,T_k)=0$ because there are no  morphisms between an  object in the tube $\mathcal{T}^0$ and an object in the component  $\mathcal{S}^0$. In consequence,  $\Hom(T_i,T_k) =0$.\\

\end{itemize}

So far we have seen that all the relations in the ideal  $I_{\Delta}$ are also relations in   $\End(F(\D))$.  To finish,  we still have to prove that there are no  other possible relations in   $\End(F(\D))$. \\

%
%
%
%
%

We begin by proving  there are not  relations of (minimal) length  greater than or equal to three. Let ($A\stackrel{\alpha}{\rightarrow} B \stackrel{\beta} {\rightsquigarrow} C\stackrel {\gamma}{\rightarrow} D$) be a relation  of length $r\geq 3$, where $\gamma\beta\alpha=0$ . The arrows
($A\rightarrow B$) and ($C\rightarrow D$) are of length 1 and the path
 ($B\rightsquigarrow C$) can have arbitrary length. In addition, suppose that the compositions
  ($A\rightarrow B \rightsquigarrow C $) and ($B \rightsquigarrow C \rightarrow D$) are nonzero. Let $d_*$  be the $m$-diagonal such that $F(d_*)=*$ with $*\in \{A,B,C,D\}$. We divide the proof into  5 cases.\\

\begin{itemize}
  \item [(a')] \underline{The 4 diagonals are of type $1$}. By lemma \ref{morfismos entre P' y algun tau de P[1]} there are no nonzero morphisms between the components $S^d$ and $S^{d+1}$. In consequence the four  diagonals must be in the same component  $S^d$ and thus in the ray $d_A^{ \ ur}$ or the ray $d_A^{ \ dr}$. Hence, the composition ($A\stackrel{\alpha}{\rightarrow} B \stackrel{\beta} {\rightsquigarrow} C\stackrel {\gamma}{\rightarrow} D$) cannot be zero, contrary to our assumption. \\

  \item [(b')] \underline{Three diagonals are of type $1$ and one is of type $2$ (or 3)}. By the argument given in the case above, the three  diagonals of type $1$ have to be in the same component $S^d$. Thus, we can assume that $d_A, d_B$ and $d_C$ are of type $1$ in $S^0$ and $d_D$ is of type $2$ in a  tube. Since the morphism ($C\rightarrow D$) is not zero, the diagonal $d_D$ shares with $d_C$ the vertex in the outer polygon. In consequence, $d_D\in T^0$ and all objects are in the same degree $0$. Hence, we can think like in the case $m=1$ and since the  $1$-cluster tilted algebras of type $\widetilde{\mathbb{A}}$ are gentle, we cannot have the  relation ($A\stackrel{\alpha}{\rightarrow} B \stackrel{\beta} {\rightsquigarrow} C\stackrel {\gamma}{\rightarrow} D )=0$.\\

  \item [(c')] \underline{Two diagonals are of type $2$ (or 3) and two are of type $1$}. Lemma \ref{morfismos entre P' y algun tau de P[1]} implies both diagonals of type $1$ are consecutive in the same component $S^d$ or one diagonal is $d_A$ and the other one is  $d_D$. \\
       To begin, suppose we are in the first case. Then, we have the following sub-cases:

      \begin{itemize}
        \item [1.] If $d_A$ and $d_B$ are in $S^0$ and $d_C$ and $d_D$ are in the same tube $T^0$, we can argue as in the case $m=1$ and conclude that it is not possible. If instead, $d_C$ is in the tube $T^0$ and $d_D$ in the tube $T^1$, the calculation done in  $(d)$ above shows that ($B \stackrel{\beta} {\rightsquigarrow} C\stackrel {\gamma}{\rightarrow} D) =0$, a contradiction.\\

        \item [2.] If $d_C$ and $d_D$ are in  $S^1$, we have the  case dual to the one considered above .\\

        \item  [3.] If $d_B$ and $d_C$ are in $S^0$ the diagonal $d_D$ has to be in  $T^1$. Hence, the computation done in  $(c)$ above implies  that $(B \stackrel{\beta} {\rightsquigarrow} C\stackrel {\gamma}{\rightarrow} D)=0$, contrary to our assumption.\\
      \end{itemize}

       Finally, suppose that $d_A$ and $d_D$ are the diagonals of type $1$. Then, we have two sub-cases:\\

       \begin{itemize}
         \item [1.] If $d_B$ and $d_C$ are in the same tube we can assume that the objects $B,C$  belong to the tube  $\mathcal{T}^{d}$ and the object $D$ belongs to the component $\mathcal{S}^{d+1}$. Suppose that $A=P_a$ and $D=\tau^r P_b[1]$. If
$(A\rightarrow B \rightsquigarrow C)=0$ we have a contradiction; otherwise, by lemma
\ref{elementos del tubo q tienen morf desde un proy}  we deduce that $B,C\in M_*^{dl}$ with $$*=\left\{
                     \begin{array}{ll}
                       p-a, & \hbox{ if $a\in\{1,\cdots, p-1\}$;} \\
                       0, & \hbox{otherwise.}
                     \end{array}
                   \right.
.$$ Since $(C\rightarrow
D)\neq 0$  lemma \ref{tubo tauideproy[1]}  implies $C\in
\tau^{r-1}M_\sharp^{dr}=M_{\sharp+r-1}^{dr}$ with $$\sharp=\left\{
                     \begin{array}{ll}
                       p-b, & \hbox{if $b\in\{1,\cdots, p-1\}$;} \\
                       0, & \hbox{otherwise.}
                     \end{array}
                   \right.
$$ Hence $\{C\}=M_{\sharp+r-1}^{dr}
\cap M_*^{dl}$. If $(B \rightarrow D)$ is not zero we would have
  $B\in M_{\sharp+r-1}^{dr} \cap M_*^{dl}= \{C\}$,
which is impossible.\\

         \item [2.] If $d_B$ and $d_C$ are in different tubes, the computation done in $(d)$ above shows that $(B \stackrel{\beta} {\rightsquigarrow} C\stackrel {\gamma}{\rightarrow} D)=0$, a contradiction.\\

       \end{itemize}

  \item [(d')] \underline{Three diagonals are of type $2$ (or 3) and one is of type $1$ }. We have the following sub-cases.

      \begin{itemize}
        \item [1.] Suppose that  $d_A$ and $d_B$ are in  $T^0$, $d_C$ is in $S^1$ and $d_D$ in $T^1$. Then we can assume that $C= \tau^r P_c[1]$ and since
$(A \rightarrow B \rightarrow C)\neq 0$,  lemma \ref{tubo tauideproy[1]} implies
$A,B\in M_{*-1+r}^{dr}$  with $$*=\left\{
                     \begin{array}{ll}
                       p-c, & \hbox{if $c\in\{1,\cdots, p-1\}$;} \\
                       0, & \hbox{otherwise.}
                     \end{array}
                   \right.
$$

 If $(B\rightsquigarrow C \rightarrow D)\neq 0$, then  the level of  $B$ is not $p-1$ and thus   $\alpha(B)\neq 0$. Lemma \ref{morfismos entre un tubo y el siguiente}  now yields  $D\in (\mathcal{B}^-(\tau
B)[1])^{dl}=(\tau \mathcal{B}^-(B)[1])^{dl}=(\tau M_{*-1+r}[1])^{dl}=(M_{*-1+r-1}[1])^{dl}$.
On the other hand, since $(C\rightarrow D)\neq 0$, we have  $\tau^{-r}D\in (M_{*}[1])^{dl}$ and consequently
$D\in (M_{*+r}[1])^{dl}\cap (M_{*+r-2}[1])^{dl} =\emptyset $ except if  $p=2$. But, for  $p=2$,  this  is not  possible, because we cannot have two factors of  $T$ in the mouth of the tube.\\

        \item  [2.] Assume that $d_A$ and $d_B$ are in $T^0$, $d_C$ in $T^1$ and  $d_D$ in $S^2$. Since the morphism $(C\rightarrow D)$ is not zero, the diagonals $d_D$ and $d_C$ share the source of $d_C$. Thus, the calculation done in  $(d)$ above gives $(B \stackrel{\beta} {\rightsquigarrow} C\stackrel {\gamma}{\rightarrow} D)=0$, a contradiction.\\

        \item [3.] Assume that the three diagonals of type $2$ are one in the tube $T^d$ and the others two in the tube $T^{d+1}$. Then, the computation done   in $(a)-1$ above gives that the composition of the morphisms between these three diagonals is zero, a contradiction.\\

        \item [4.] Suppose that each diagonal of type $2$ is in a  different tube. Then, for example,  $d_A\in T^0$, $d_B\in T^1$ and $d_C\in T^2$. Now  the calculation done in  $(a)-2$, shows that $(A\rightarrow B \rightsquigarrow C) =0$, which contradicts our assumption.\\

        \item [5.] Suppose that $d_A$  belongs to $T^0$, $d_B$  belongs to $T^1$, $d_C$  to $S^2$ and $d_D$ to $T^2$. Then, the three diagonals $d_A,d_B$ and $d_C$ are in the situation of the case $(d)$ above and, in  consequence, the composition $(A\rightarrow B \rightsquigarrow C) =0$, a contradiction.\\

        \item [6.] Suppose that the three  diagonals $d_A,d_B$ and $d_C$ are in the  tube $T^0$ and $d_D$ is in the component $S^1$. Then, we can think like  in the case $m=1$ because all the diagonals are in the same degree. Since we know that for $m=1$ there are not relations of length greater than or equal to $3$ we finish.\\

      \end{itemize}

  \item [(e')] \underline{The four diagonals are of type $2$ (or $3$)}. Let us see the  possibles sub-cases which are not among the ones considered  in $(d')$ for the three diagonals of type $2$.

      \begin{itemize}
        \item [1.] $d_A$ and $d_B$ are in the tube $T^0$ and $d_C$ and $d_D$ are in the tube $T^1$. Then, we have two options for the objects $C$ and $D$. Either $D\in C^{ur}$, or $D\in C^{dr}$. Suppose that  $D\in C^{ur}$. Let $C'$ and $D'$ be objects such that $C'[1]=C$ and $D'[1]=D$, then $D'\in C'^{ur}$. Since ($A\rightarrow C$) and ( $B\rightarrow D$) are nonzero, the $(m+1)$-CY property of $\C$ gives that $0\neq \Hom(A,C'[1])\cong D \Hom(C',\tau A)$ and $0\neq \Hom(B,D'[1])\cong D \Hom(D',\tau B)$. Thus, $C'\in (\tau A)^-$ and $D'\in (\tau B)^-$. The same computation applied to the zero  morphism ($A\rightarrow D$) gives $D'\in ((\tau A)^-)^c$. On the other hand, since there is a nonzero morphism ($A\rightarrow B$) the sources of $d_A$ and $d_B$ must be the same vertex. Hence, $B\in A^{dr}$ and equivalently $\tau B\in (\tau A)^{dr}$. In consequence, $D'\in C'^{ur}\cap ((\tau B)^-\setminus (\tau A)^-)=\emptyset$, which is not possible. To conclude, assume that $D\in C^{dr}$, then $s(d_C)=s(d_D)$ and the computation done in the previous part $(a)$ yields that $(B \rightsquigarrow C \rightarrow D)=0$. \\

        \item [2.]  $d_A$, $d_B$  and $d_C$ are in the tube $T^0$ and $d_D$ is in the tube $T^1$. Note that $A,B$ and $C$ are in the same ray, that is,  the ray  $C^{dl}$ or $A^{dr}$. If $A,B$ and $C$ belong to the ray
$C^{dl}$ then $\mathcal{B}^-(\tau C)\neq \mathcal{B}^-(\tau B)$
and since  $(C\rightarrow D)\neq 0$, $D\in (\mathcal{B}^-(\tau
C)[1])^{dl}\neq (\mathcal{B}^-(\tau B)[1])^{dl} $. In consequence,  lemma \ref{morfismos entre un tubo y el siguiente} implies that $(B\rightarrow D)=0$.

If instead  $A,B$ and $C$ belong to the ray  $A^{dr}$ we have that
$\mathcal{B}^-(\tau C)=\mathcal{B}^-(\tau A)$. In addition, since we have the morphisms  $A\rightarrow B \rightarrow C$ in the same ray  the level of  $A$ cannot be $p-1$, and then $\alpha(C)< \alpha(A)\neq 0$.   Since  $(C\rightarrow D)\neq 0$ we have   $D\in \bigl((\mathcal{B}^-(\tau C)[1])^{dl}\bigr)_{\leq \alpha(C)} \ \subset
\bigl ((\mathcal{B}^-(\tau A)[1])^{dl}\bigr)_{\leq \alpha(A)}$. In consequence,  lemma \ref{morfismos entre un tubo y el siguiente} implies that   $(A\rightsquigarrow D)\neq 0$, a contradiction.\\

      \end{itemize}

\end{itemize}

Note that in all cases except the case (a'), the two last sub-cases of (c') and the sub-case  (e')-2 we did not really have to use that the composition $A\rightarrow B\rightsquigarrow C \rightarrow D$ is zero and we have actually proved that one of the  compositions $A\rightarrow B\rightsquigarrow C$ or $B\rightsquigarrow C \rightarrow D$ is zero.\\

Next,  we claim that there are no relations of the form

\[\xymatrix@R=0.4pc@C=4pc{
& B_1 \ar[r]^{g_1} & C_1\ar[ddr]^{   f_1}&\\
& B_2 \ar[r]^{g_2} & C_2\ar[dr]^{f_2 \hspace{.5cm}}&\\
A \ar[uur]^{f_1'}\ar[ur]^{ \hspace{.6cm}   f_2'}\ar[dr]^{f_t'}&\vdots & \vdots & D\\
& B_t \ar[r]^{g_t} & C_t\ar[ur]^{f_t}&\\
}\]

where a linear combination of  compositions  $f_i g_i f_i'$ is $0$,   $t\geq 2$ and the paths $g_i$ can have arbitrary lengh. Since the idea is to prove that all the compositions $f_i g_i f_i'$ except at most one are zero,  we can assume  $t=2$.
The last remark  implies that if the diagonals $d_A,d_{B_1},d_{C_1}$ and $d_D$  are not in the three cases considered above we cannot have such a relation. Then, it just remains to see these three cases.\


\begin{itemize}

  \item [(i)] \underline{Case (a')}. In this case the diagonals $d_A,d_{B_1},d_{C_1}$ and $d_D$ are of type $1$. By lemma \ref{morfismos entre P' y algun tau de P[1]} every morphism from $S^d$ to $S^{d+1}$ is zero. In consequence the four diagonals have to be in the same component $S^d$ and so in the ray $d_A^{ \ ur}$ or the ray $d_A^{ \ dr}$. The same argument applied to the  diagonals $d_A,d_{B_2},d_{C_2}$ and $d_D$ implies that the  four diagonals have to be in the ray $d_A^{ \ ur}$ or the ray $d_A^{ \ dr}$. Thus, $D$ must be the last  object in the two rays  $A^{ ur}$ and  $A^{dr}$ and we can assume that $A=P_p$ and $D=P_0$. However, if  this is the case then we cannot have the required relation.\\


  \item [(ii)] \underline{The two last sub-cases of case (c')}. Since the composition $A\rightarrow D$ cannot be zero, we do not have these sub-cases.\\

  \item  [(iii)] \underline{Case (e')-2}. In this case $d_A$, $d_{B_1}$  and $d_{C_1}$ are in the tube $T^0$ and $d_D$ is in the tube $T^1$. We can assume that $A, B_1$ and  $C_1$ are in the ray $A^{ur}$.  Then, $B_2$ and $C_2$ have to lie on the ray $A^{dr}$. Hence, the computation done in   (e')-2 shows that the  composition $B_1\rightarrow D$ is zero, which is impossible. \\
\end{itemize}

Finally, there only remains to see that we cannot have relations of the  form

\[\xymatrix@R=0.4pc@C=5pc{
& B_1\ar[ddr]^{f_1}&\\
& B_2\ar[dr]_{f_2}& \\
A \ar[uur]^{f_1'}\ar[ur]_{f_2'}\ar[dr]_{f_t'}& \vdots & D\\
& B_t\ar[ur]_{f_t}&\\
}\]

where the sum of the compositions  $f_i f_i'$ has to be $0$ and $t\geq 2$.  The idea is to show that one of the  compositions $f_i f_i'$ has to be zero. Then, we can assume that  $t=2$.

%

 At first, we observe that if the four  factors $A,B_1,B_2$ and $D$ of $F(\D)$ are in the same component $\mathcal{S}^i$ the computation done  in  $(i)$ implies that we cannot have a commutativity relation. Next, observe that if  $A\in \mathcal{S}^0$ and   $B_1,B_2,D\in \mathcal{T}^0_p$ the three objects $B_1,B_2,D$ have to be in the same ray and consequently the endomorphisms algebra of  $F(\D)$ does not contains the required   relation. Finally, since every morphism from    $S^0$ to  $S^1$ is zero, the only possibility is to have $A,B_1,B_2\in \mathcal{S}^0$ and $D\in \mathcal{T}^0_p$. But then we can think like in the case  $m=1$  where we  cannot have such a relation. See \cite[Proposition 3.1]{Murphy2010}.\\

 Summarizing,  we have seen that the only possible relations in the  algebra $\End(F(\D))$ are the relations given by the ideal  $I_{\Delta}$.
\end{proof}

%
%
%
%
%
%
%
%
%
%
%
%
%
%

\section{The  $m$-cluster tilted algebras of type $\tilde{\mathbb{A}}$}

The aim of this section  is to show that  $m$-cluster tilted algebras of type $\tilde{\mathbb{A}}$ are gentle and to find a  characterization of their bound quivers.
Murphy in  \cite{Murphy2010} showed that   $m$-cluster tilted algebras of type $\mathbb{A}$ are gentle for any   $m\geq 1$. We want to prove the same for the  type  $\tilde{\mathbb{A}}$. For $m=1$   the result was showed in  \cite{ABCP09}. For any $m\geq 2$ is a easily consequence of proposition \ref{el carcaj con relaciones de End(F(delta))}.

\begin{prop}\label{son amables} The $m$-cluster-tilted algebras of type $\tilde{\mathbb{A}}$    are gentle for any $m\geq 2$.
\end{prop}

\begin{proof}  The result follows from considerations of the possible divisions of $\P$.  The following figures make the required properties clear.

%

\vspace*{-.2cm}

%

\vspace*{-.4cm}
 $$ \begin{array}{cc}
    \hspace*{-1.8cm} \includegraphics[scale=.6]{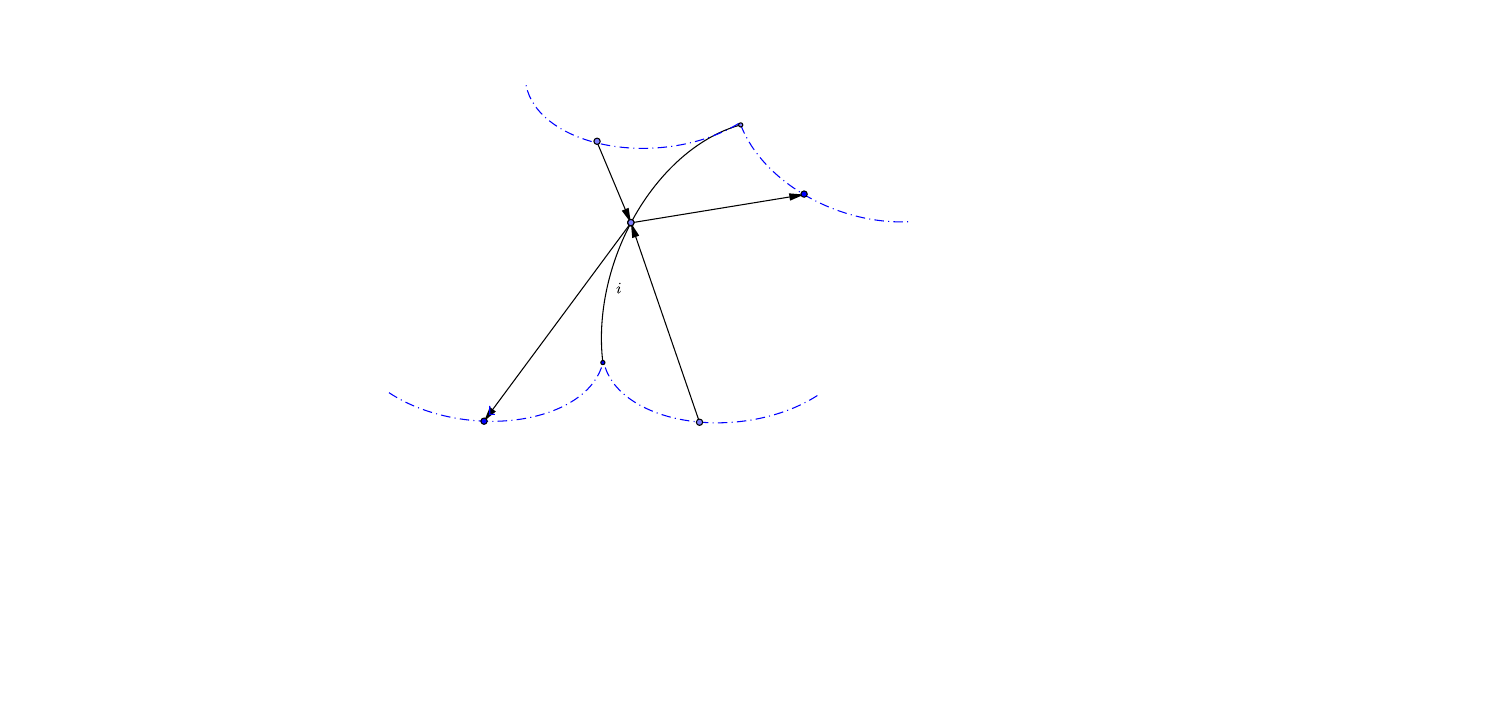} & \hspace*{-3.3cm}\includegraphics[scale=.5]{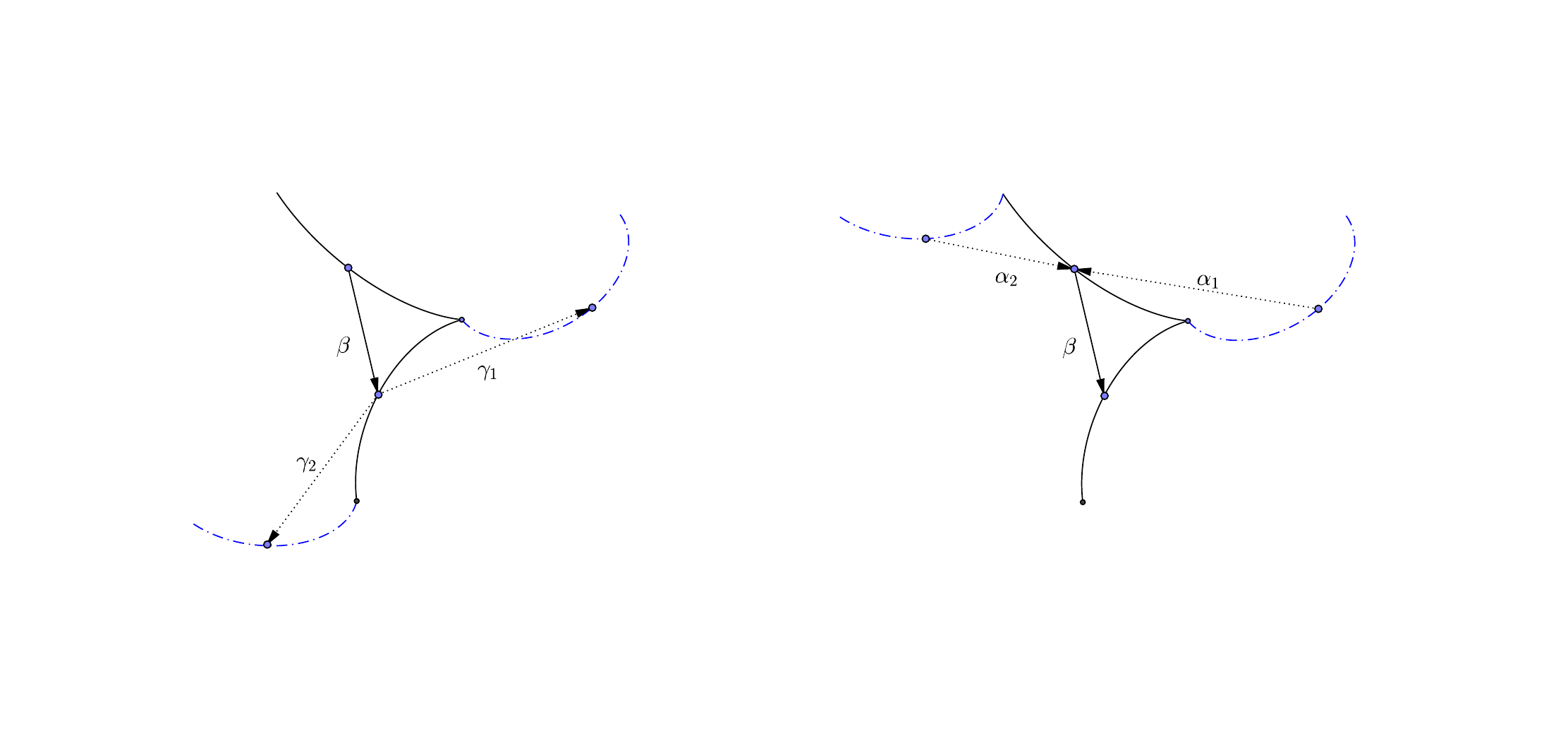} \\
\end{array}$$


%



\vspace*{-.8cm}
Finally observe that by definition,  the ideal $I_{\D}$ is generated by monomial quadratic relations.
\end{proof}


%
%
%

Given a bound quiver $(Q,I)$ and an integer $m$, a cycle is called \emph{ $m$-saturated} if it is an oriented cycle consisting of $m+2$ arrows such that the composition of any two consecutive arrows on this cycle belongs to $I$. Recall that two relations  $r$ and $r'$ in the bound quiver $(Q,I)$ are said to be   \textit{consecutive} if there is a walk  $v=wr=r'w'$ in  $(Q,I)$ such that $r$ and $r'$ point in the same direction and share an arrow.\\


For the following definition we fix a natural $m\geq2$.

\begin{defi}
 Let $\widetilde{\mathcal{C}}$ be a cycle without  relations (oriented or not) and fix an orientation of its arrows.  We say that   $A\cong \K Q/I$ is an  \textit{algebra with root $\widetilde{\mathcal{C}}$} if its bound quiver can be constructed  as follows:

\begin{enumerate}
\item We add to the  cycle $\widetilde{\mathcal{C}}$ gentle quivers in such a way that the final quiver remains gentle and connected. These gentle quivers that we add can have   cycles all of which are    $m$-saturated. We call these quivers  \textit{rays}.
\item We can add  relations to the  cycle $\widetilde{\mathcal{C}}$. If the cycle  $\widetilde{\mathcal{C}}$ is oriented  we must add at least one relation.\\
\end{enumerate}

Also, we will refer to the  cycle $\widetilde{\mathcal{C}}$ as the  \textit{root cycle}.
\end{defi}

In the sequel let  $\widetilde{\mathcal{C}}$  be a  non-saturated cycle and $A\cong \K Q/I$  an algebra with root $\widetilde{\mathcal{C}}$.

 \begin{defi}
 Let $c$ be a vertex (which is not in an  $m$-saturated cycle) in a ray. We said that $c$ is the  \textit{union vertex} of the ray, if $c$ lies also in the root cycle.
 \end{defi}

%
%
%
%
%
%
%

\begin{obss}\textnormal{Let $\widetilde{\mathcal{C}}$ be a  cycle and $A$ an algebra with root  $\widetilde{\mathcal{C}}$.}
 \begin{enumerate}
            \item \textnormal{Each ray of $A$ can share with the cycle $\widetilde{\mathcal{C}}$  at most  $m+2$ vertices. If it shares just one vertex, this vertex is the union vertex of the ray. If it shares more than one vertex, the ray and the cycle $\widetilde{\mathcal{C}}$ are connected through an  $m$-saturated cycle.}

           \item \textnormal{ For each union vertex there is at least one relation involving at least one arrow of $\mathcal{C}$. } \\ 
          \end{enumerate}

\end{obss}

\begin{defi}
Let $a$ be an union vertex and  let $\rho$ be the involved relation in the root cycle. The  relation $\rho$ is called:

   \begin{itemize}
     \item [a)] \textit{internal union relation} of the ray if both arrows of the relation belong to the root cycle.
     \item [b)] \textit{external union relation} of the ray if just one arrow of the relation belongs to the root cycle.
   \end{itemize}

\end{defi}

%
%
%
%

Now, we  assign an orientation to each union relation.  Let $\rho$ be an union relation. We say that  $\rho$ is
\textit{clockwise ( or counterclockwise)} if the involved arrows which lie in the root cycle are clockwise (or counterclockwise, respectively) oriented.


\begin{ejem}
Let $A$ be the algebra given by the following bound quiver.

\vspace*{-.4cm}
\begin{figure}[H]
\includegraphics[scale=.5]{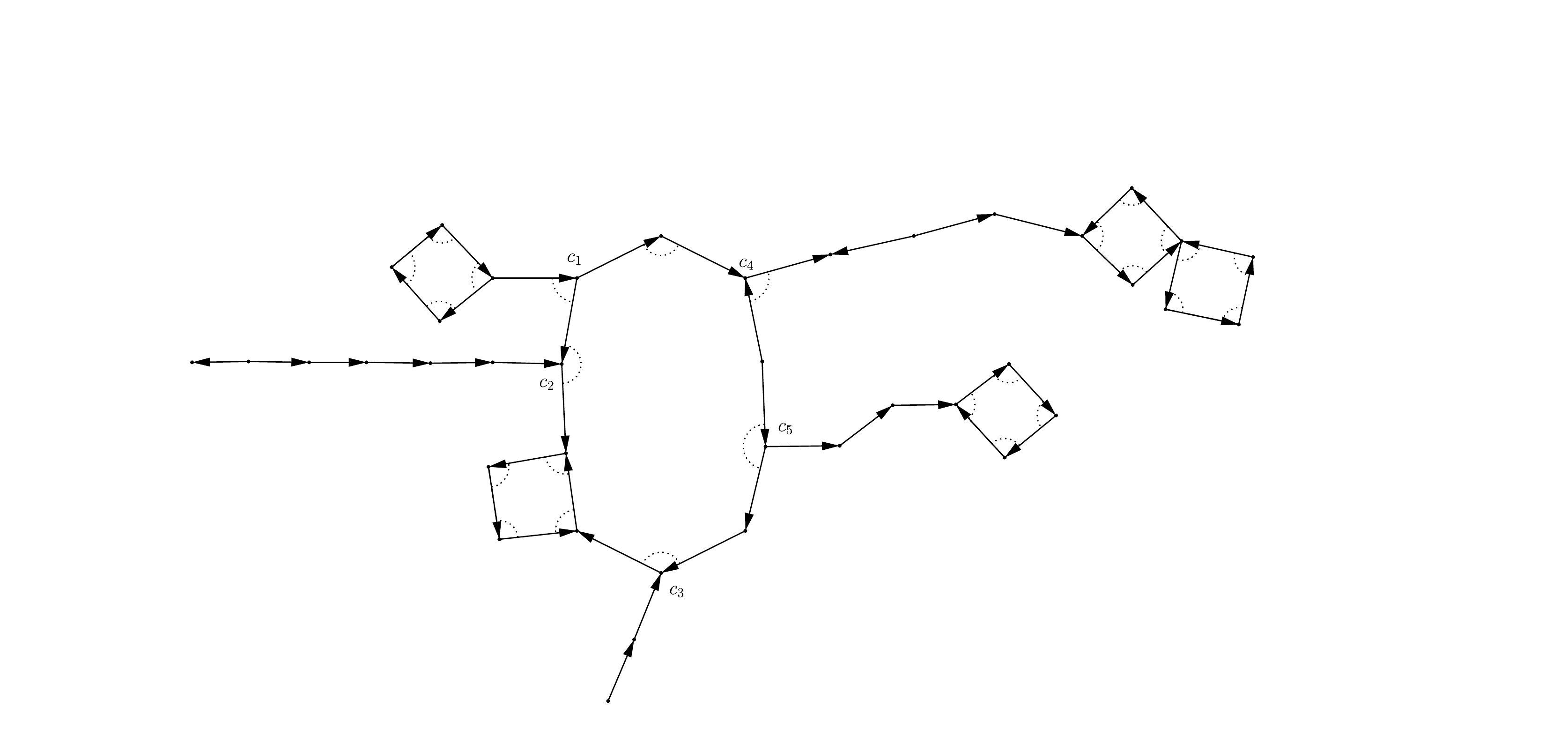}
\end{figure}

The relations at the vertices $c_1$  and $c_4$ are  counterclockwise external union relations. The one at the vertex $c_2$ is a counterclockwise internal union relation. Finally, the relations at the vertices $c_3$ and $c_5$ are clockwise internal union relations.


\end{ejem}

We collect the following easy consequences of proposition \ref{el carcaj con relaciones de End(F(delta))}.

\begin{obss}\label{no existen masde m-1 relaciones seguidas Atilde}
\textnormal{ Let $A\cong \K Q/I$ be a connected  $m$-cluster tilted algebra of type $\tilde{\mathbb{A}}$, then:}
\textnormal{\begin{enumerate}
       \item $(Q,I)$ does  not contain non-saturated  cycles or else has exactly one non-saturated  cycle $\widetilde{\mathcal{C}}$  in such a way that $A$ is an algebra with root $\widetilde{\mathcal{C}}$.
              \item In the first case  the only possible cycles are $m$-saturated.
              \item If $m\neq 1$, we can have relations outside  of  $m$-saturated cycles, but with the following restriction: we can have at most $m-1$ consecutive relations outside of an $m$-saturated cycle .
               \end{enumerate}}
\end{obss}

In particular the previous remark implies that we do not have uniqueness of type for $m$-cluster tilted algebras because every  $m$-cluster tilted algebra of type $\tilde{\mathbb{A}}$
without a non-saturated cycle  is at the same time an  $m$-cluster tilted algebras of type $\mathbb{A}$. The following example ilustrate this remark. \\

\begin{ejem}
Let $(Q,I)$ be a bound  quiver associated to the following  $4$-angulation $\D$ of  $P_{2,2,2}$.
\begin{figure}[H]
\begin{center}
\includegraphics[scale=.6]{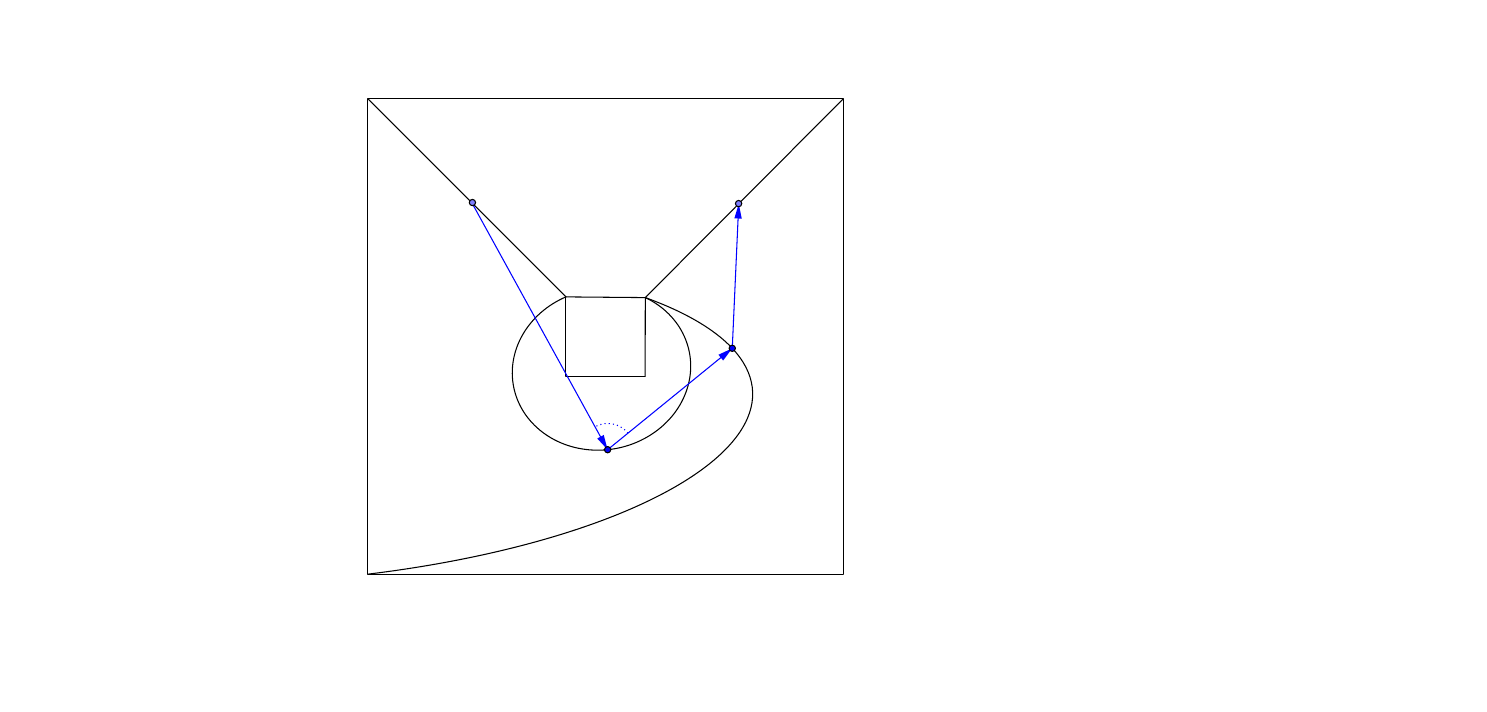}
\end{center}
\end{figure}

 Then the algebra $\K Q/I$ is a $2$-cluster tilted algebra of  type $\widetilde{\mathbb{A}}_{2,2}$ and also of  type $\mathbb{A}_{4}$.

\end{ejem}
Since  $m$-cluster tilted algebras of  type $\mathbb{A}$ are well understood, \cite{Murphy2010, BB10, GubBust} we are going to focus in $m$-cluster tilted algebras of  type $\tilde{\mathbb{A}}$ having at least one non-saturated cycle and so having (at least) a root cycle. In this case we choose one and we fix it. We start with  the relations in the root cycle.


 \begin{defi}
Let  $\rho$ be a relation in the root cycle.  The  relation $\rho$ is said     \textit{
strictly internal}  if the two arrows of the relation belong to the root cycle, but not to any   $m$-saturated cycle. In addition, we say that  the strictly internal relation is \textit{clockwise (or counterclockwise)} if both involved arrows  are  clockwise (or counterclockwise, respectively) oriented.
\end{defi}

Let $\alpha_h$  be the number of  clockwise strictly internal relations and  $\alpha_a$ the number of  counterclockwise strictly internal relations.\\

\begin{obs}
Let  $\mathcal{C}$ be an  $m$-saturated cycle sharing at least two vertices  $i$ and $j$  with the root cycle. Then, the root cycle orientation induces an orientation on the arrows of   $\mathcal{C}$. That is that there are   $k$ arrows between the vertices $i$ and $j$ clockwise oriented and $m+2-k$ arrows counterclockwise oriented.
\end{obs}

Given an  $m$-saturated cycle   $\mathcal{C}$ having at least two vertices in common with the root cycle  denote by $\beta_a(\mathcal{C})$ the number of arrows of $\mathcal{C}$ with counterclockwise orientation  less one and $\beta_h(\mathcal{C})$ the number of arrows of $\mathcal{C}$ with clockwise orientation  less one.

Observe that for any such   cycle $\mathcal{C}$, we have  $\beta_a(\mathcal{C})+\beta_h(\mathcal{C})=m$.\\

\begin{lema}\label{ciclo q comparte mas de una flecha entonces hay flechas - 1 relaciones en el sentido contrario}
Let  $\mathcal{C}$ be an   $m$-saturated cycle as in the above remark. Then at least one of the following conditions hold.

 \begin{enumerate}
   \item There are at least $k-1=\beta_h(\mathcal{C})$ strictly  internal counterclockwise relations in the root cycle.
   \item There are at least $m-k+1=\beta_a(\mathcal{C})$ strictly  internal clockwise relations in the root cycle.
   \item There is another  $m$-saturated cycle $\mathcal{C'}$  with $\beta_h(\mathcal{C}')=\beta_a(\mathcal{C})$ and $\beta_a(\mathcal{C}')=\beta_h(\mathcal{C})$.
 \end{enumerate}
        .
\end{lema}

\begin{proof}
Suppose that we have a  cycle $\mathcal{C}$ with  $k$ arrows  between the vertices  $i$ and $j$ clockwise oriented and  $m+2-k$ arrows counterclockwise oriented. Then, in the corresponding $\m$-angulation we have the following situation where to simplify the  notation we label the vertices of the outer and inner polygon modulo $m$.

\begin{figure}[H]
\begin{center}
\includegraphics[scale=.7]{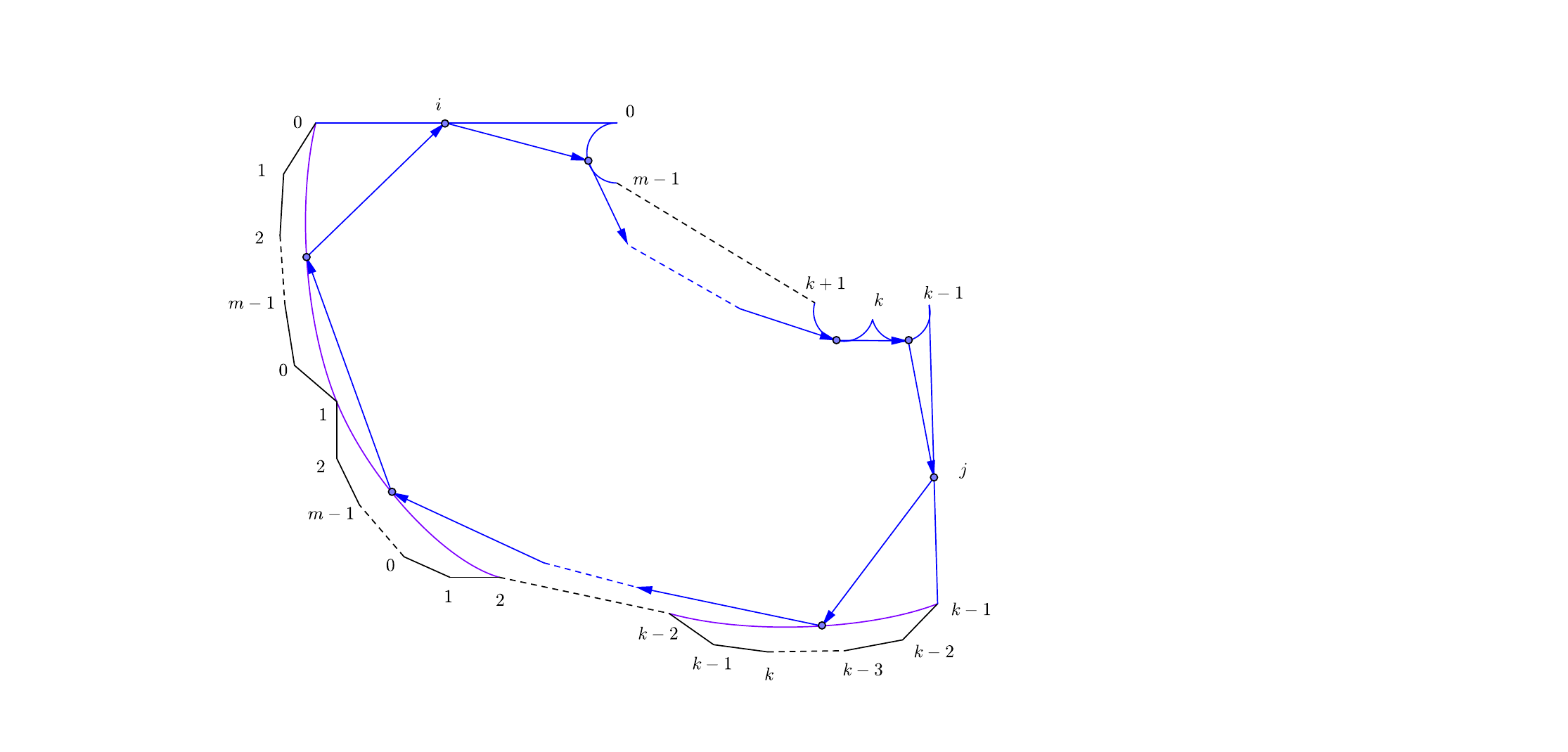}
\end{center}
\end{figure}

If we want to have a non-saturated  cycle involving the vertices $i$ and $j$, we must find a  partition from the region of  $\P$ delimited by the  diagonal $i$, the edge of the outer polygon  lying between the vertices  $0$ and $k-1$ clockwise, the diagonal $j$ and the edge of the inner polygon   lying between the vertices  $k-1$ and $0$ counterclockwise.

If we  add an arrow to the non-saturated cycle we have the following  sub-angulations.

 $$ \begin{array}{cc}
    \hspace*{-2cm} \includegraphics[scale=.8]{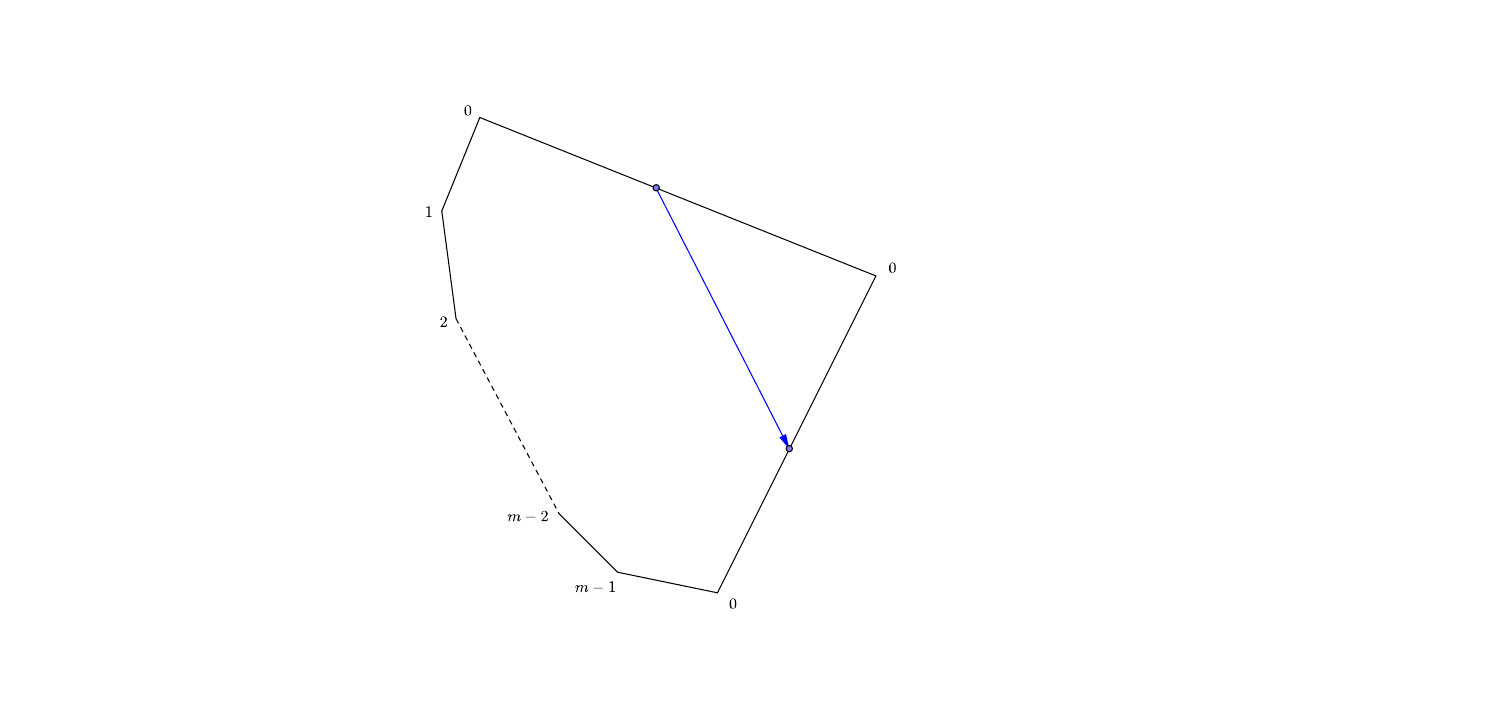} & \hspace*{-4cm}\includegraphics[scale=.8]{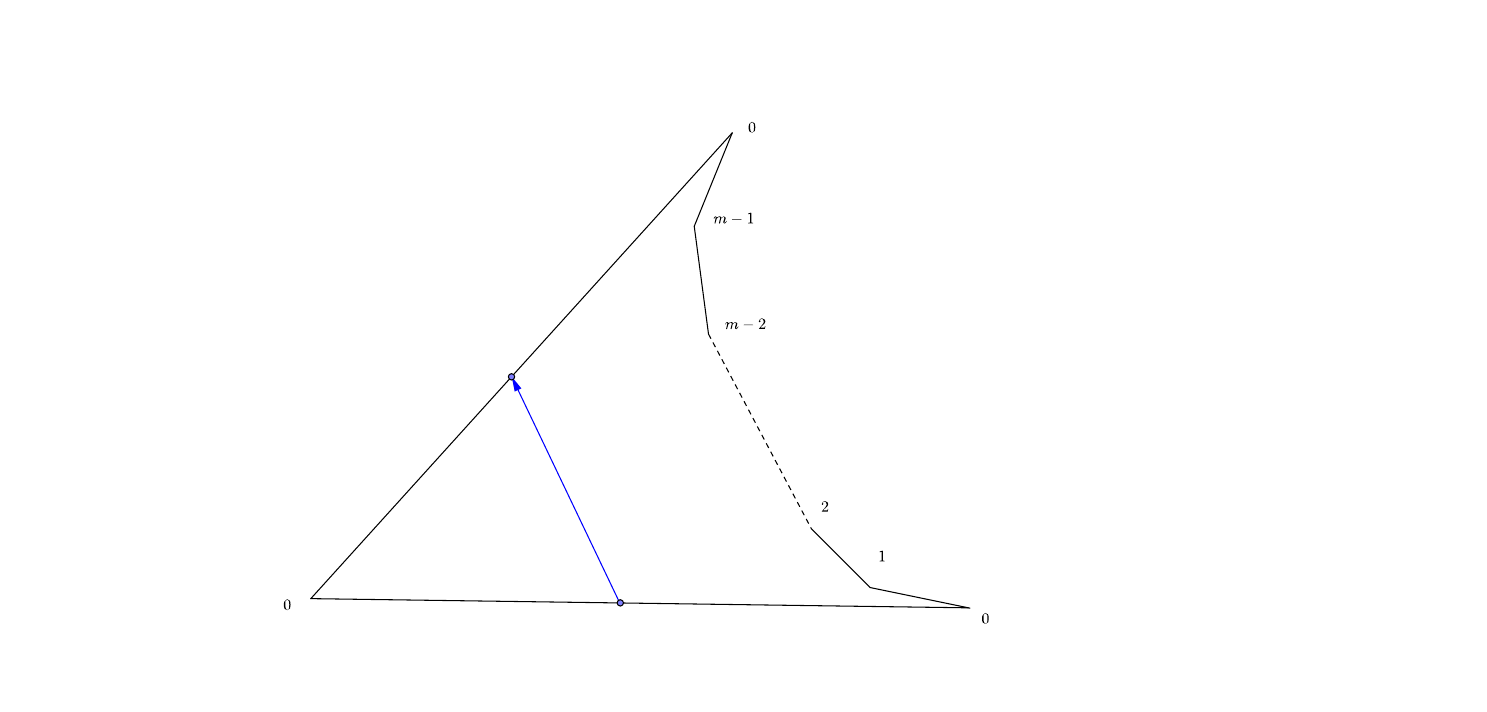} \\
   \hspace*{-3cm} \text{counterclockwise arrow}  & \hspace*{-8cm} \text{clockwise arrow}  \\
    &\\

  \end{array}$$

Then, when we add arrows the number modulo $m$  of vertices of the outer and inner polygon  do not change. We have the same situation when we add an  $m$-saturated cycle sharing just one arrow with the root cycle (the arrow between the vertices $a$ and $b$). See the figure below.

 \begin{figure}[H]
 \begin{center}
 \includegraphics[scale=.6]{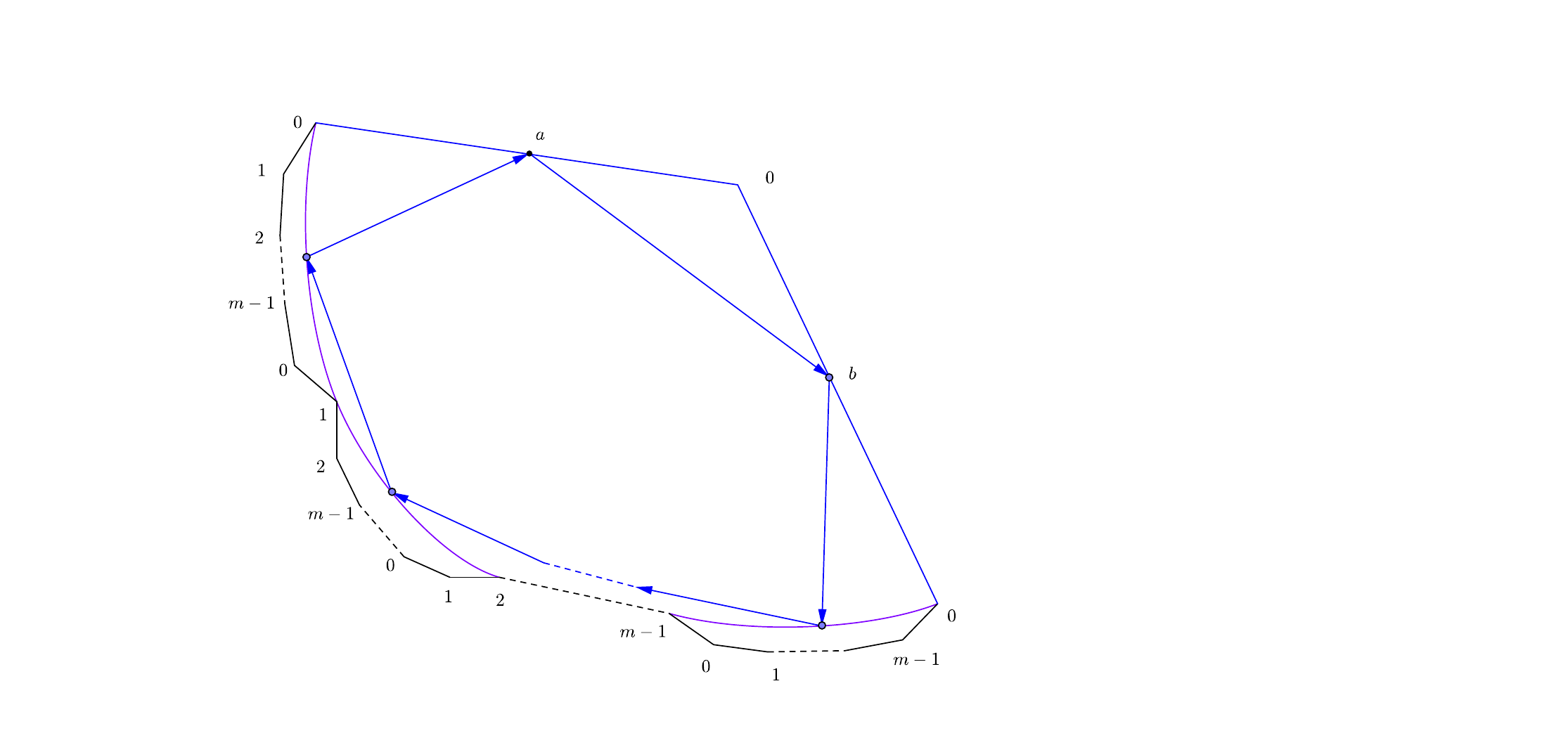}
    \end{center}
  \end{figure}

\vspace{-.7cm} \hspace*{3.5cm}   counterclockwise $m$-saturated cycle

\vspace{.8cm}

If we add a clockwise  relation we have to add one to the  value of the outer polygon vertex and subtract one to the value of the inner polygon vertex. See the following figure.

 \begin{figure}[H]
 \begin{center}
 \includegraphics[scale=.7]{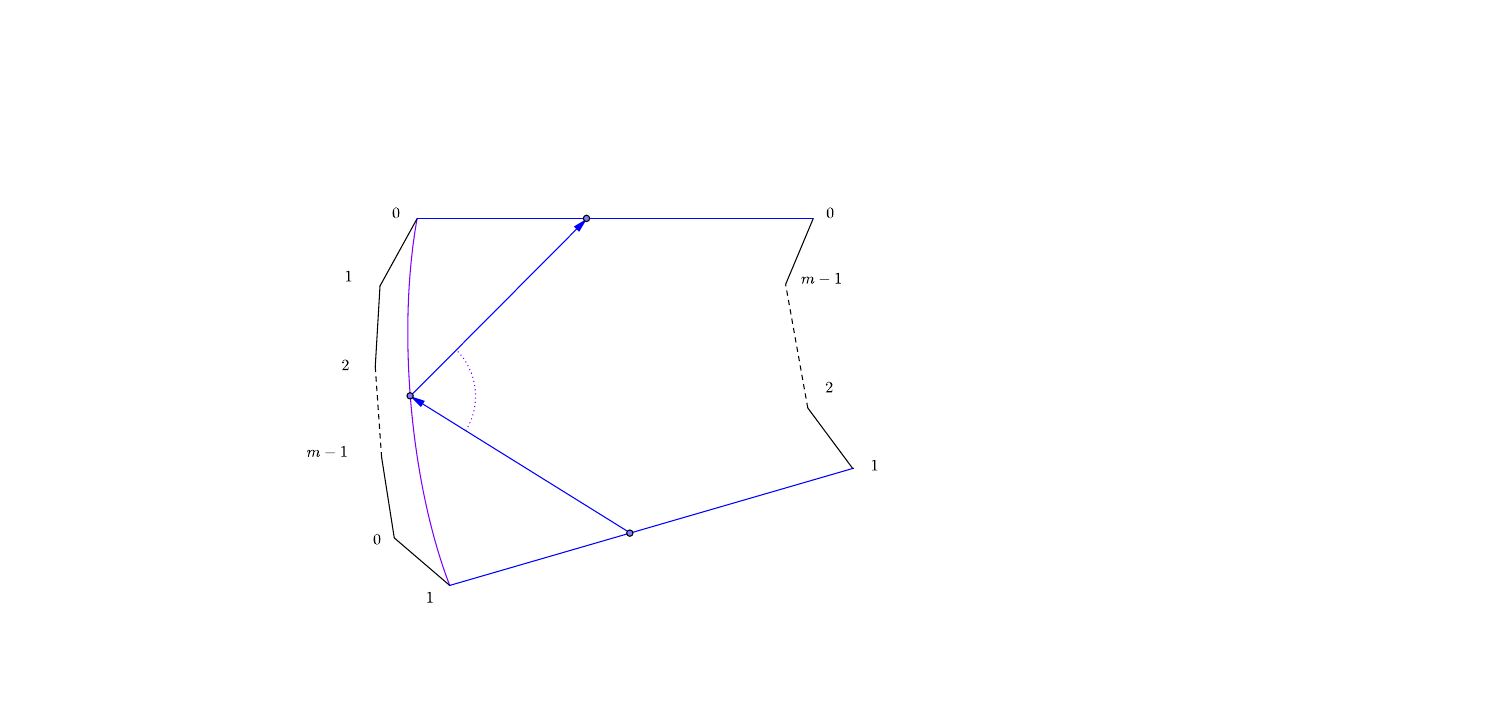}
 \end{center}
 \end{figure}

\vspace{-.8cm} \hspace*{5cm}   clockwise relation

\vspace{.8cm}

 We have the dual case if we add one counterclockwise relation. That is, we subtract one to the value of the outer polygon vertex and we add one  to the value of the inner polygon vertex.

Then, if we add $m-k+1$ clockwise relations the vertex of the outer  polygon is going to take the value  $k-1+ m-k+1=m\equiv 0$ (modulo $m$), and  the vertex of the inner  polygon is going to take the value  $k-1- (m-k+1)\equiv k-1 -(k-1)\equiv 0$ (modulo $m$).

It follows from these considerations that the only ways to "close" the  non-saturated (root) cycle  which contains the vertices  $i$ and $j$ is to add $m-k+1$  internal relations in the clockwise sense, or $k-1$ relations in the counterclockwise sense, or another  $m$-saturated cycle with $k$ arrows in the counterclockwise sense and $m-k+2$ arrows in the clockwise sense.
\end{proof}

The lemma above allows us to make the following definition.

\begin{defi}
Let  $\mathcal{C}$ be an   $m$-saturated cycle sharing at least two vertices with the root cycle. We say that the  cycle $\mathcal{C}$ is  \textit{clockwise} (\textit{or counterclockwise}) if the  condition $1$ ( or $2$, respectively) holds.  If $3$ holds or  $1$ and $2$ simultaneously hold, we say that $\mathcal{C}$ is clockwise if $\beta_h(\mathcal{C})\leq \beta_a(\mathcal{C})$ or counterclockwise otherwise.
\end{defi}

We fix the following  notation, let $\mathfrak{C}_h$ be the set of clockwise   $m$-saturated cycles and $\mathfrak{C}_a$ the set of counterclockwise  $m$-saturated cycles.\\


\begin{defi}
We define the \textit{number of clockwise internal relations} $r_h$ as

$$r_h=\alpha_{h}+\sum _{\mathcal{C}\in \mathfrak{C}_h}  \beta_h(\mathcal{C}) $$

 In the same way we define the number of counterclockwise internal relations  $r_a$ by

$$r_a=\alpha_{a}+\sum _{\mathcal{C}\in \mathfrak{C}_a}  \beta_a(\mathcal{C})$$

\end{defi}

\begin{ejem}
Let $\mathcal{C}$ be the $3$-saturated cycle of the following bound quiver and fix as  root cycle the cycle of length $10$.

\begin{figure}[H]
\begin{center}
\hspace{1cm}\includegraphics[scale=.5]{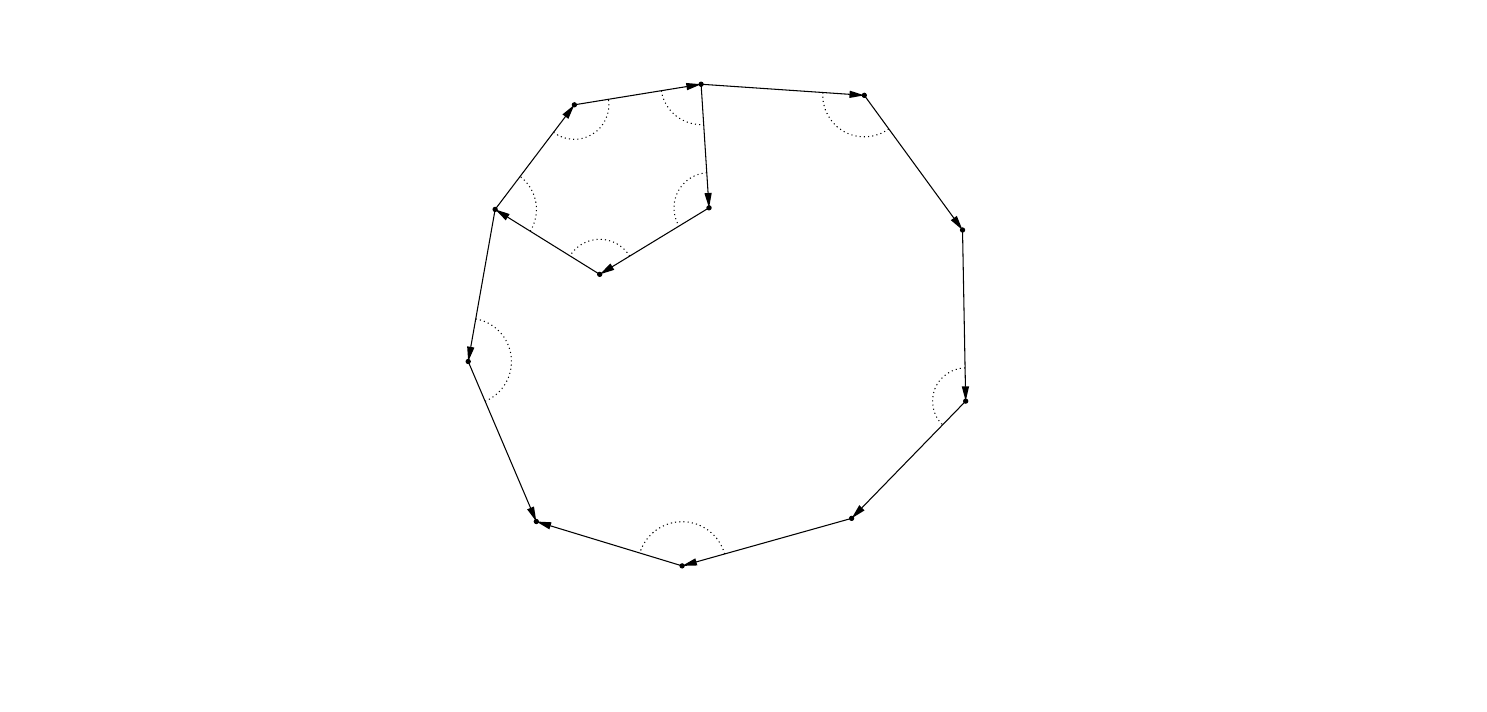}
\end{center}
\end{figure}


Then, $\beta_h(\mathcal{C})=1$ and $\beta_a(\mathcal{C})=2$. Since $\alpha_h=3>\beta_a(\mathcal{C})$, $\alpha_a=1=\beta_h(\mathcal{C})$ and $\beta_h(\mathcal{C})\leq \beta_a(\mathcal{C})$ ; the cycle $\mathcal{C}$ is counterclockwise.\

The number of clockwise internal relations is  $r_h= 3$ and the number of counterclockwise internal relations  is $r_a= 1+2=3$. Then $r_h\equiv r_a$ modulo $m=3$.

\end{ejem}


In general, we have the following property.\\

\begin{prop}\label{internal relations} Let $A$ be a connected  $m$-cluster tilted algebra of  type $\tilde{\mathbb{A}}$ with a root cycle. If there are internal relations on the root cycle, then the  number modulo $m$ of clockwise internal relations is equal to the number of counterclockwise internal  relations. That is,  $r_h\equiv r_a$ modulo $m$.

\end{prop}

\begin{proof}
To begin, assume that the set $\mathfrak{C}_a\cup \mathfrak{C}_h$ is empty. Then, the  discussion in lemma \ref{ciclo q comparte mas de una flecha entonces hay flechas - 1 relaciones en el sentido contrario} says that each time we add a clockwise oriented relation, we must in order to close the root cycle, add also a counterclockwise oriented relation. By the same lemma it also  follows that if we add  a multiple of $m$  relations in any sense, then the values of the vertices of the outer or inner polygons do not change. Now, assume there is a cycle $\mathcal{C}$ in $\mathfrak{C}_a$. Then, there are $\beta_h(\mathcal{C})$ counterclockwise strictly internal relations. Therefore, $\alpha_a=\alpha_a'+\beta_h(\mathcal{C})$ with $\alpha_a'\equiv \alpha_h$ modulo $m$. In consequence, $r_a=\alpha_a'+\beta_h(\mathcal{C})$ and $r_h=\alpha_h+ \beta_h(\mathcal{C})$. Clearly $r_h\equiv r_a$ modulo $m$. Similar considerations apply if we have an arbitrary number of  clockwise or counterclockwise $m$-saturated cycles.
\end{proof}

Finally, summarizing the results obtained about the bound quiver of an
 $m$-cluster tilted algebra of type $\widetilde{\mathbb{A}}$ we get our main theorem.

\begin{teo}
Let $(Q,I)$ be a connected bound quiver. Then $(Q,I)$ is the bound quiver of a connected component of an  $m$-cluster tilted algebra $A$ of type $\widetilde{\mathbb{A}}$ if and only if  $(Q,I)$ is a gentle  quiver such that:
\begin{itemize}
  \item [(a)] It can contain a non-saturated  cycle  $\widetilde{\mathcal{C}}$ in such a way that $A$ is an algebra with root $\widetilde{\mathcal{C}}$.
  \item [(b)] If it contains more cycles, then all of them are  $m$-saturated cycles.
  \item [(c)] Outside of an $m$-saturated cycle it can have at most  $m-1$ consecutive relations.
  \item [(d)] If  $\widetilde{\mathcal{C}}$ is an oriented  cycle, then it must have at least one internal relation.
  \item [(e)] If there are internal relations in the root cycle, then   the number  of clockwise oriented  relations   is equal modulo $m$ to the number of counterclockwise oriented.

\end{itemize}
\end{teo}

\begin{proof}
Observe that $(Q,I)$ is gentle for proposition \ref{son amables}, $(a),(b)$ and $(c)$ follow from remark \ref{no existen masde m-1 relaciones seguidas Atilde}; $(d)$ is the second condition for being an algebra with root and finally $(e)$ is the statement of proposition \ref{internal relations}.
The converse statement may be proved in much the same way as lemma \ref{ciclo q comparte mas de una flecha entonces hay flechas - 1 relaciones en el sentido contrario}. See the following example to ilustrate the procedure.
\end{proof}

\begin{ejem} Let $(Q,I)$ be the following bound quiver.

\[\xymatrix@R=0.3pc@C=.3pc{  &&& h \ar[rrr] &  \ar@{.}@/^/[lld] && a \ar[rrrddd] &&& \\
&&& &&& &&& \\
&&& &&& &&& \\
g \ar[uuurrr] \ar[ddd] &&&  &&&  &&& b \\
&&& &&& &&& \\
&&& &&& &&&   \\
f  \ar[dddrrr] &&& &&& &&& c \ar[uuu]\\
&&& &&& && \ar@{.}@/^/[ruu] & \\
&&& &&& &&& \\
&&& e &&& d \ar[rrruuu] \ar[lll] &&&}\]


We want to find an $\m$-angulation of a certain polygon $P_{p,q,m}$. We start with the diagonal corresponding to the vertex $a$. Since the arrow $a\rightarrow b$ is clockwise oriented and in $a$ there is not involved any relation, $a$ is in correspondence with a diagonal of type $1$, say the diagonal $d_a$ between the vertices $O_0$ and $I_0$ and $b$ is in correspondence with another diagonal of type $1$, label  $d_b$ between the vertices $O_0$ and $I_{m}$. In $c$ there is a counterclockwise internal relation, then $c$ is in correspondence with the diagonal of type $3$ between the vertices $I_m$ and $I_{2m+1}$ and $d$ is in correspondence with the diagonal $d_d$ of type $1$ between the vertices $O_{3m+1}$ and $I_{2m+1}$. Since in $e$ there is not a relation and the arrow $d\rightarrow e$ is clockwise oriented, $e$ is in correspondence with a diagonal $d_e$ of type $1$ between the vertices $O_{2m+1}$ and $I_{3m+1}$. Analogously, since  in $f$ there is no a relation and the arrow $g\rightarrow f$ is counterclockwise oriented, $g$ is in correspondence with a diagonal $d_g$ of type $1$ between the vertices $O_{m+1}$ and $I_{3m+1}$. Finally, since in $h$ there is a clockwise oriented relation, $h$ is in correspondence with a diagonal of type $2$ between the vertices $O_{m+1}$ and $O_0$. Therefore we get the following $\m$-angulation of $P_{4,4,m}$ which proves that $(Q,I)$ is the bound quiver of a connected   $m$-cluster tilted algebra of type $\widetilde{\mathbb{A}}_{4,4}$.

\begin{figure}[H]
\begin{center}
\includegraphics[scale=.5]{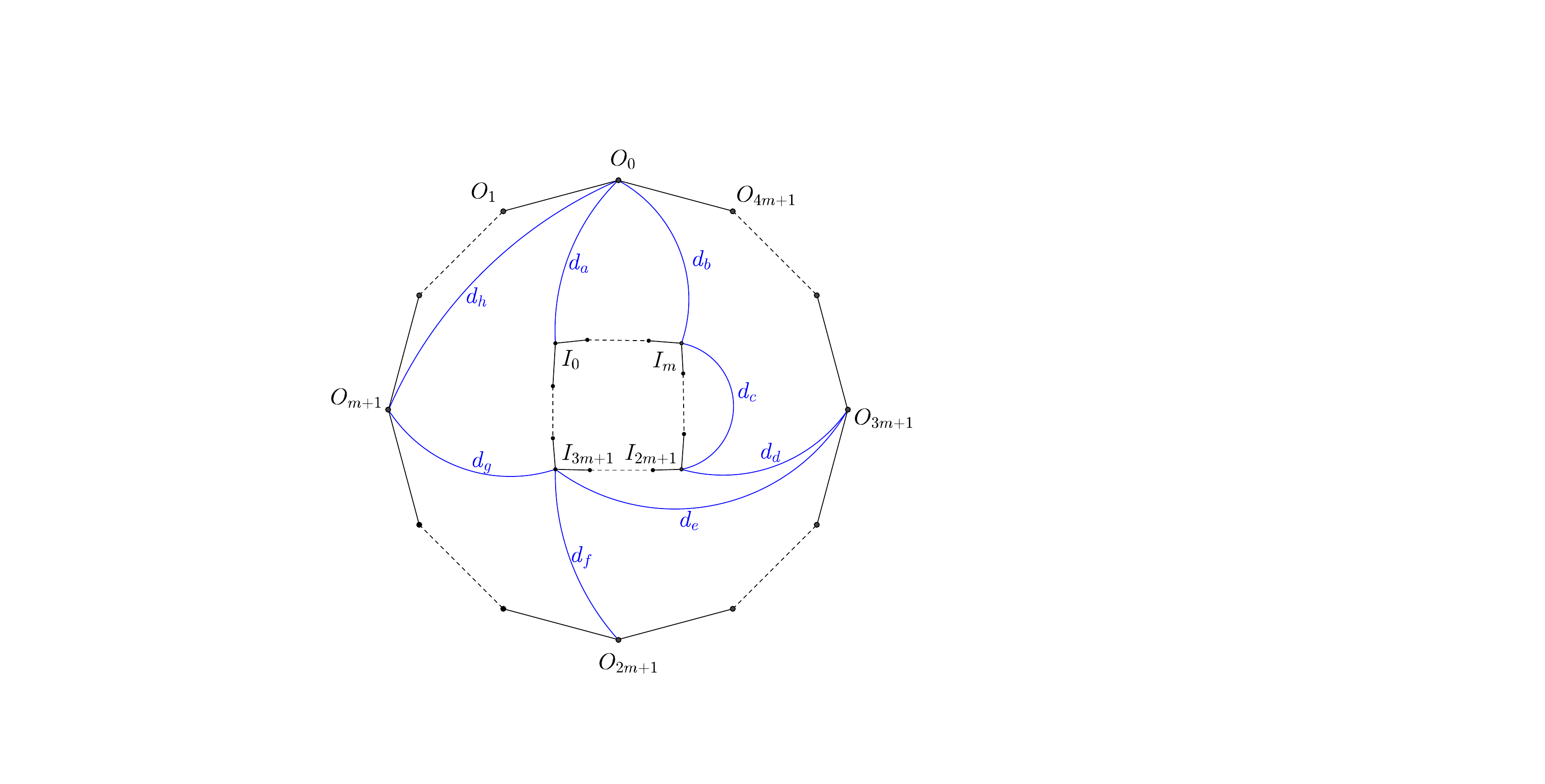}
\end{center}
\end{figure}

 \end{ejem}

%
%
%
%
%

\section*{Acknowledgements}
The author gratefully thanks Ibrahim Assem for several interesting and helpful discussions. This paper is part of the  author's Ph.D. thesis, done under the direction of Ibrahim Assem, she gratefully acknowledges financial support from the Faculty of Sciences of the Universit\'e de Sherbrooke, and  the  \emph{Agencia Nacional de Investigaci\'{o}n e Innovaci\'{o}n (ANII)} of Uruguay.

\bibliographystyle{acm}

\bibliography{library}

\end{document}